\documentclass[11pt, oneside]{amsart}
\usepackage[latin1]{inputenc}
\usepackage[english]{babel}
\usepackage{amsmath,amssymb,amsthm,amsfonts}
\usepackage{a4}
\usepackage{latexsym,dsfont}
\usepackage{enumerate}
\usepackage{graphicx}
\usepackage{latexsym}
\usepackage{color}

\makeatletter
\@namedef{subjclassname@2010}{%
  \textup{2010} Mathematics Subject Classification}
\makeatother

\newtheorem{lemma}{Lemma}[section]
\newtheorem{proposition}[lemma]{Proposition}
\newtheorem{theorem}[lemma]{Theorem}
\newtheorem{corollary}[lemma]{Corollary}
\newtheorem{remark}[lemma]{Remark}
\newtheorem{definition}[lemma]{Definition}

\numberwithin{equation}{section}
\DeclareMathOperator{\supp}{supp}

\begin{document}
\title{Dual spaces to Orlicz-Lorentz spaces}
\author[Kami\'nska]{Anna Kami\'nska}
\address[Kami\'nska]{Department of Mathematics\\
University of Memphis\\
Memphis, USA}
\email{\texttt{kaminska@memphis.edu}}
\author[Le\'snik]{Karol Le\'snik}
\address[Le\'snik]{Institute of Mathematics of Electric Faculty\\
Pozna\'n University of Technology, ul. Piotrowo 3a, 60-965 Pozna\'{n}, Poland}
\email{\texttt{klesnik@vp.pl}}
\author[Raynaud]{Yves Raynaud}
\address[Raynaud]{Institut de Math\'ematiques de Jussieu, Universit\'e Paris 06-UPMC and CNRS, 4 place Jussieu, F-75252 Paris cedex 05, France}
\email{\texttt{yves.raynaud@upmc.fr}}

\begin{abstract}
For an Orlicz function $\varphi$ and a decreasing weight $w$, two intrinsic exact
descriptions are presented for the norm in the Köthe dual of an Orlicz-Lorentz function space
$\Lambda_{\varphi,w}$ or a sequence space $\lambda_{\varphi,w}$, equipped with either Luxemburg or Amemiya norms.
The first description of the dual norm is given via the modular $\inf\{\int\varphi_*(f^*/|g|)|g|: g\prec w\}$, where $f^*$ is the decreasing rearrangement  of $f$,  $g\prec w$ denotes the submajorization of $g$ by $w$ and $\varphi_*$ is the complementary function to $\varphi$. The second one is stated in terms of the modular $\int_I \varphi_*((f^*)^0/w)w$, where $(f^*)^0$ is Halperin's level function of $f^*$ with respect to $w$. That these two descriptions are equivalent results from the identity  $\inf\{\int\psi(f^*/|g|)|g|: g\prec w\}=\int_I \psi((f^*)^0/w)w$ valid for any measurable function $f$ and Orlicz function $\psi$. Analogous identity and  dual representations are also presented for sequence spaces.

\end{abstract}
\maketitle


%\keywords{Orlicz-Lorentz spaces, dual spaces, level function, Calder\'on-Lozanovskii spaces, r.i. spaces}

%\subjclass[2010]{42B25, 46B10, 46E30}

\section{\protect\bigskip Introduction}

The main goal of the paper is to give an isometric description of the Köthe
dual space of Orlicz-Lorentz space $\Lambda_{\varphi,w}$, where $\varphi$ is an Orlicz function and $w$ is a decreasing locally integrable weight function.  The Orlicz-Lorentz spaces have been studied extensively for the past two decades, since  when their basic properties were established in \cite{Kam90}. So far however there have not been given satisfactory isometric description of the dual spaces of Orlicz-Lorentz spaces. There are several
different isomorphic representations of the K\"othe dual spaces $(\Lambda_{\varphi,w})^{\prime}$ given
for example in \cite{HuKaMa02} or in \cite{KaMa07}.    In \cite{ChCHW01} is posted an unsolved problem number XIV asking for finding an isometric representation of $(\Lambda_{\varphi,w})^{\prime}$.

Orlicz-Lorentz spaces can be treated  as a special case of more general Calder\'on-Lozanovskii spaces. Lozanovskii in his paper \cite{Lo78b} (see also \cite{Lo65}-\cite{Lo78a}, \cite{Ra97} and \cite{Re88}) proved a duality theorem, which in particular can be applied to Orlicz-Lorentz spaces. However his original formulas are too general and not explicit enough  for applications in the setting of Lorentz type spaces. Here we show that  Lozanovskii's formulas for dual norms and the K\"othe dual spaces  can be expressed in terms of the recently introduced modular $P_{\varphi,w}$ and the corresponding modular space $\mathcal{M}_{\varphi,w}$ (see \cite{KR12}). In fact $\mathcal{M}_{\varphi,w}= \{f\in L^0: P_{\varphi,w}(\lambda f) < \infty \ \text{for some}\ \lambda>0\}$, where $L^0$ is the space of Lebesgue measurable real functions on $I=[0,\alpha)$ and
\begin{equation*}
P_{\varphi ,w}\left( f\right) =\inf \Big\{ \int_I \varphi \left(
f^{\ast }/|g|\right) |g|: g\prec w \Big\}.
\end{equation*}
The notation $g\prec w$ means  that $g$ is  submajorized by $w$, that is $\int_{0}^{t}g^*\leq
\int_{0}^{t}w$ for all $t\in I$.

 In the case when $\varphi(u) = u^p$, $1<p<\infty$, and $\Lambda_{\varphi,w}$ becomes a classical Lorentz space $\Lambda_{p,w}$, a different  explicit isometric description of its dual was given by Halperin in \cite{Ha53}.  He introduced the notion of level intervals and level functions with respect to the weight $w$, and applied them to obtain the formula of the norm of dual space. Here we study the level functions and the modulars in the environment of Orlicz-Lorentz spaces, which allows us to extend Halperin's theorem to the case of those spaces.

  Consequently we give in this paper two different isometric representations of dual spaces of Orlicz-Lorentz spaces, one by means of submajorization by the weight $w$, and another one by level functions with respect to $w$. They are valid for both function and sequence spaces.

The paper is organized as follows. In section 1 we give basic notations and notions needed further. Among others we define Calder\'on-Lozanovskii spaces and Orlicz-Lorentz spaces equipped with standard Amemiya and Luxemburg norms.

In section 2 we recall the definition of function spaces $\mathcal{M}_{\varphi,w}$ and then applying the general duality theorem of Lozanovskii, we prove that the K\"othe dual space $(\Lambda_{\varphi,w})'$ is $\mathcal{M}_{\varphi_*,w}$ with equality of corresponding norms. In case when the space $\Lambda_{\varphi,w}$ is separable, it is also an isometric representation of its dual space.  This representation is given for both Amemiya and Luxemburg norms.

Section 3 is devoted to a number of specific properties of the modular $P_{\varphi ,w}(f)$.
There is  given a sequence of technical results that leads to the main theorem describing an algorithm for calculation of infimum in the formula of the   modular $P_{\varphi ,w}(f)$ when $f$ is a simple decreasing function.  This is Theorem \ref{th:3} which states that the function $g^f$ produced by
Algorithm A is minimizing the modular $P_{\varphi ,w}(f)$. It is interesting to observe that $g^f$ depends only on $f$ and $w$, but not on $\varphi$.

In section 4 we give another isometric representation of the K\"othe dual spaces using the so called level functions $f^0$ with respect to $w$ that Halperin introduced in \cite{Ha53}.   Applying the results of the previous section, in particular Algorithm A, we first prove that $P_{\varphi,w} (f) = \int_I\varphi(f/g^f) g^f = \int_\varphi(f^0/w) w$ for a decreasing simple function $f$. In the next step we extend this result to any $f\in \Lambda_{\varphi,w}$, which in fact yields the second duality theorem. Theorem \ref{final theorem} summarizes all K\"othe duality formulas for function spaces $\Lambda_{\varphi,w}$ equipped with     either Amemiya  or Luxemburg norms. Halperin's duality result for spaces $\Lambda_{p,w}$, $1<p<\infty$,   is then a corollary from Theorem \ref{final theorem}.

In the last fifth section we present the analogous results for the Orlicz-Lorentz sequence spaces $\lambda_{\varphi,w}$. We show first that sequence spaces as well as their K\"othe dual spaces can be embedded isometrically into appropriate Orlicz-Lorentz function spaces. Next applying the results of the previous sections for function spaces we quickly obtain the analogous isometric representations of the dual spaces of $\lambda_{\varphi,w}$ in terms of  the  sequence spaces $\mathfrak{m}_{\varphi_*,w}$ introduced in \cite{KR12} as well as in terms of the spaces generated by  $\varphi_*$, $w$ and level sequences.

Let us agree first on the notation and basic notions used in the paper.  By $\varphi $ we denote an {\it Orlicz function}, that is $\varphi: [0,\infty) \to
[0,\infty)$, $\varphi(0)= 0$, $\varphi$ is convex and $\varphi$ is strictly
increasing. Let $\varphi_{\ast }$ be the {\it complementary}  %\cite{Lo78a}
function to $\varphi $, that
is $\varphi_*(s) = \sup_{t\ge 0} \{st - \varphi(t)\}$, $s\ge 0$. By $\varphi^{-1}$ denote the inverse function to $\varphi$.
It is said that $\varphi$ is an {\it $N$-function} whenever $\lim_{t\to 0+}
\varphi(t)/t =0$ and $\lim_{t\to\infty}\varphi(t)/t = \infty$. It is well known that $\varphi_*$ is an $N$-function whenever $\varphi$ is such a function \cite{KrRut}. Recall also that the function $t\mapsto \varphi(a/t)t$ is decreasing and convex on $\mathbb{R}_+$ for every $a>0$. The first fact results from the well known property that the function $\varphi(t)/t$ is increasing for $t>0$, while for the second one we have by convexity of $\varphi$ that for any $t_1, t_2 \ge 0$,
%\begin{align*}
\[
\varphi \Big( \frac{2a}{t_{1}+t_{2}}\Big) \frac{t_{1}+t_{2}}{2}
=\varphi \Big( \frac{at_{1}}{( t_{1}+t_{2}) t_{1}}+\frac{at_{2}%
}{ ( t_{1}+t_{2}) t_{2}}\Big) \frac{t_{1}+t_{2}}{2}
\leq \frac{t_{1}}{2}\varphi \Big( \frac{a}{t_{1}}\Big) +\frac{t_{2}}{2}%
\varphi \Big( \frac{a}{t_{2}}\Big).
\]
It also shows that $t\mapsto \varphi \left( a/t\right) t$ is
strictly convex if $\varphi $ is strictly convex. It is said that $\varphi$ satisfies a {\it $\Delta_2$-condition} for all arguments, respectively for large arguments, wehnever $\varphi(2u)\le K\varphi(u)$ for all $u\ge 0$, respectively for all $u\ge u_0$ and some $u_0 \ge 0$.

Given an Orlicz function $\varphi$, define  its associated {\it Calder\'on-Lozanovskii function} as
\begin{equation}
\rho \left( t,s\right) =\rho _{\varphi }\left( t,s\right) =\varphi
^{-1}\left( s/t\right) t,\ \ \ s\geq 0, \ t>0,  \label{def ro}
\end{equation}%
and the {\it conjugate function} to $\rho$ as
\begin{equation*}
\hat{\rho}\left( t,s\right) =\hat{\rho}_{\varphi }\left( t,s\right)
=\inf_{u,v>0}\frac{us+vt}{\rho \left( u,v\right) },\ \ \ \ s,t\geq 0.
\end{equation*}%
It is well known that the function $\rho_\varphi(t,s)$ is concave on $%
(0,\infty) \times [0,\infty)$. Moreover, if $\varphi $ is an $N$-function then
\begin{equation}
\hat{\rho}_{\varphi }\left( t,s\right) =\varphi _{\ast }^{-1}\left(
t/s\right) s, \ \ \ t\ge 0,\  s>0, \label{repr of conj}
\end{equation}
(see Example 3 in \cite{Lo78a}, or Example 7 in \cite{Re88}).
%(KL: Equality (\ref{repr of conj}) is announced in \cite{Lo78b}, in
%Berezhnoi's paper and in \cite{Ra97} but only for N -functions. When I was
%writing my mgr thesis I couldn't find the proof so I proved it, but only for
%N -functions too. It seems that it is not proved for each functions
%directly, but it should be a consequence of Lozanovskii's duality thm. and
%representation of dual of Orlicz spaces, because $L_{\varphi _{\ast
%}}^{o}\equiv L_{\varphi }^{\prime }\equiv \left( \rho _{\varphi ,\infty
%}\left( L^{1},L^{\infty }\right) \right) ^{\prime }\equiv \hat{\rho}%
%_{\varphi ,1}\left( L^{1},L^{\infty }\right) $, where $L_{\varphi _{\ast
%}}^{o}$ is Orlicz space with the Orlicz norm.)

Let further $I=[0,\alpha )$ where $0<\alpha \leq \infty $. By $L^{0}$ denote
the set of all Lebesgue measurable real-valued functions on $I$. Given $f\in L^{0}$ define its {\it distribution function} as
\begin{equation*}
d_{f}(\lambda )=\mu \{ t\in I:\left\vert f(t)\right\vert >\lambda
\},\ \ \ \lambda \geq 0,
\end{equation*}%
and its {\it decreasing rearrangement} $f^{\ast }$ as
\begin{equation*}
f^{\ast }(t)=\inf \{ \lambda >0:d_{f}(\lambda )\leq t\} ,\ \ \ t\in I.
\end{equation*}
Here by decreasing or increasing functions we mean the functions which are non-incresing or non-decreasing, respectively.
We say that $f\in L^0$ is {\it submajorized} by $g\in L^0$ and we write
\begin{equation*}
f\prec g \ \ \ \text{whenever}\ \ \ \int_0^t f^* \le \int_0^t g^* \ \ \
\text{for every} \ \ t\in I.
\end{equation*}
For any decreasing locally integrable function $h$ let further incorporate
the following notation
\begin{equation*}
H(t) = \int_0^t h, \ \ \ t\in I.
\end{equation*}
A Banach space $(E, \|\cdot\|_E)$ is called a {\it Banach function space} (or a
{\it K\"othe space}) if $E\subset L^0$ and whenever $f\in L^0$, $g\in E$ and $%
|f|\le |g|$ a.e. then $f\in E$ and $\|f\|_E \le \|g\|_E$. We will also assume that each Banach function space contains a weak unit, i.e. there is $f\in E$ such that $f(t)>0$ for a.a. $t\in I$. By $E'$ denote the K\"othe dual space to $E$, which consists of all $f\in L^0$ such that $\|f\|_{E'} = \sup\left\{\int_Ifg: \|g\|_E \le 1\right\} < \infty$. The space $E'$ equipped with the norm $\|\cdot\|_{E'}$ is a Banach function space. It is well known that $E'$ is non-trivial and contains a weak unit  \cite[Ch.15, §71, Theorem 4(a)]{Z}.

Given a  Calder\'on-Lozanovskii function $\rho $ and a couple of Banach
function spaces $E,F$, the {\it Calder\'on-Lozanovskii space} is defined as
\begin{equation*}
\rho ^{0}( E,F) =\{ f\in L^{0}:\left\Vert f\right\Vert
_{\rho }^{0}=\inf \left\{ \left\Vert g\right\Vert _{E}+\left\Vert
h\right\Vert _{F}:\left\vert f\right\vert =\rho \left( |g|,|h|\right)
\right\} <\infty \},
\end{equation*}%
\begin{equation*}
\rho (E,F)=\{f\in L^{0}:\Vert f\Vert _{\rho }=\inf \{\max (\Vert g\Vert
_{E},\Vert h\Vert _{F}):\left\vert f\right\vert =\rho (|g|,|h|)\}<\infty
\}.
\end{equation*}%
Recall that the spaces $\rho_\varphi ^{0}( L^{\infty},L^{1 }) $ and $\rho_\varphi
( L^{\infty },L^{1}) $ coincide isometrically with the Orlicz space $L^\varphi$ equipped
with its {\it Amemiya} and {\it Luxemburg norm} respectively \cite{Lo78b}.
Moreover, in the above definitions
one may take equivalently $\left\vert f\right\vert \leq \rho \left(
|g|,|h|\right) $ instead of $\left\vert f\right\vert =\rho \left(
|g|,|h|\right) $. In fact it is enough to apply Lemma 1 from \cite{Re88},
which states that if $\left\Vert g\right\Vert _{E}\leq 1,\left\Vert
h\right\Vert _{F}\leq 1$ such that $0\leq f\leq \rho \left( g,h\right) $,
then there exist $0\leq g_{1}\leq g$, $0\leq f_{1}\leq f$ satisfying $f=\rho
\left( g_{1},h_{1}\right) $. It is also known \cite{Ma89} that
\begin{eqnarray*}
\left\Vert f\right\Vert _{\rho \left( E,F\right) } &=&\inf \{
C>0:\left\vert f\right\vert \leq C\rho \left( |g|,|h|\right) ,\left\Vert
g\right\Vert _{E}\leq 1,\left\Vert h\right\Vert _{F}\leq 1\} \\
&=&\inf \{ C>0:\left\vert f\right\vert =C\rho \left( |g|,|h|\right)
,\left\Vert g\right\Vert _{E}\leq 1,\left\Vert h\right\Vert _{F}\leq
1\}.
\end{eqnarray*}
Both spaces $\rho (E,F)$ and $\rho ^{0}\left( E,F\right) $ coincide as sets
and the norms $\Vert \cdot \Vert _{\rho }$, $\Vert \cdot \Vert _{\rho }^{0}$
are evidently equivalent. The spaces $\hat{\rho}(E,F),$ $\hat{\rho}%
^{0}\left( E,F\right) $ are defined analogously as $\rho (E,F)$ and $\rho
^{0}\left( E,F\right) $ where the function $\rho $ is replaced by $\hat{\rho}
$. Moreover, the notation $\rho _{\varphi }\left( E,F\right) $ stands for the
function $\rho $ that is defined by $\varphi $ according to formula (\ref{def ro}%
).

Let $w$ be a {\it weight function} on $I$ that is $w\in L^{0}$, $w$ is positive and
decreasing on $I$, and locally integrable, i.e.
$W\left( t\right) =\int_{0}^{t}w \,<\infty$, $t\in I$.
Denote $W(\infty )=\int_{0}^{\infty }w$ in case when $\alpha = \infty$.
The \textit{Lorentz space} $\Lambda _{w}$ is classically defined as
\begin{equation*}
\Lambda _{w}=\Big\{ f\in L^{0}:\left\Vert f\right\Vert _{\Lambda
_{w}}=\int_{I}f^{\ast } w=\int_{I}f^{\ast
}\, dW <\infty \Big\},
\end{equation*}%
and the \textit{Marcinkiewicz space} $M_{W}$ as
\begin{equation*}
M_{W}=\Big\{ f\in L^{0}:\left\Vert f\right\Vert _{M_{W}}=\sup_{t\in I}
\Big(\int_0^t f^{\ast}/W(t)\Big) <\infty \Big\}.
\end{equation*}
Both are Banach function spaces and each one is the K\"othe dual space of the other one.
Note that $\| f\|_{M_W}\leq 1$ if and only if  $f\prec w$.
Let $\varphi $ be an Orlicz function and $w$ a decreasing weight function on
$I$. Then the \textit{Orlicz-Lorentz space} is the set
\begin{equation*}
\Lambda _{\varphi ,w}=\{ f\in L^{0}:\exists _{\lambda >0}\ I_{\varphi
,w}(\lambda f)<\infty \} ,
\end{equation*}%
where $I_{\varphi,w}(f)=\int_{I}\varphi (f^{\ast })w$.
It is equipped with either the Luxemburg norm
\begin{equation*}
\Vert f\Vert _{\Lambda }=\inf \{ \epsilon >0:I_{\varphi ,w}(f/\epsilon
)\leq 1\} ,
\end{equation*}%
or the Amemiya norm
\begin{equation*}
\Vert f\Vert _{\Lambda }^{0}=\inf_{k>0} \frac{1}{k}(
1+I_{\varphi ,w}(kf)) .
\end{equation*}%
By $\Lambda _{\varphi ,w}$ we denote the Orlicz-Lorentz space
equipped with the Luxemburg norm $\Vert \cdot \Vert _{\Lambda }$, and by $%
\Lambda _{\varphi ,w}^{0}$ this same space equipped with the Amemiya norm $%
\Vert \cdot \Vert _{\Lambda }^{0}$. The Orlicz-Lorentz spaces are  Calder\'on-Lozanovskii spaces
relative to the couple $(L^\infty,\Lambda_w)$, and the following
identities
\begin{equation}\label{eq:L}
\Lambda _{\varphi ,w}=\rho _{\varphi
}(L^{\infty },\Lambda _{w}),   \ \ \ \
\Lambda _{\varphi ,w}^{0} =\rho
_{\varphi }^{0}(L^{\infty },\Lambda _{w})
\end{equation}%
hold true with equalities of norms. The first equality may be found in  \cite{Ma89} (cf. \cite{HuKaMa96} and \cite{HuKaMa02}). As for the second one letting $f\in \rho ^0_{\varphi }\left( L^{\infty },E\right) $, we have
\begin{eqnarray}\label{eq:4}
\left\Vert f\right\Vert _{\rho _{\varphi }\left( L^{\infty },E\right) }^{0}
%&=&\inf \left\{ \left\Vert x\right\Vert _{E}+\left\Vert y\right\Vert
%_{L^{\infty }}:|f|=\rho _{\varphi }\left( |y|,|x|\right) \right\}\\ \notag
&=&\inf \left\{ \left\Vert x\right\Vert _{E}+\left\Vert y\right\Vert
_{L^{\infty }}:|x|=|y|\varphi (|f|/|y|) \right\} \\ \notag
&=&\inf_{k>0}\left\{ \inf \left\{ \left\Vert |y|\varphi (|f|/|y|) \right\Vert _{E}+\left\Vert y\right\Vert _{L^{\infty }}:\left\Vert
y\right\Vert _{L^{\infty }}=k\right\} \right\} \\ \notag
&=&\inf_{k>0}\left\{ \left\Vert k\varphi \left( |f|/k\right)
\right\Vert _{E}+k\right\} =\inf_{k>0}\frac{1}{k}\left(\left\Vert \varphi (
k|f|\right) \right\Vert _{E}+1) ,
\end{eqnarray}%
where the third equality is a consequence of the inequality $|y|\leq \left\Vert
y\right\Vert _{L^{\infty }}$ and the monotonicity of the function $s\mapsto
s\varphi \left( {a}/{s}\right) $. The desired equality follows for $E=\Lambda _{w}$.

\section{\protect\bigskip The dual space of an Orlicz-Lorentz space %$\Lambda_{%
%\protect\varphi,w}$}
}

In this section we will show that the K\"othe dual spaces to the Orlicz-Lorentz spaces $\Lambda_{\varphi,w}$ and $\Lambda^0_{\varphi,w}$ coincide isometrically with the spaces $\mathcal{M}^0_{\varphi_* ,w}$ and $\mathcal{M}_{\varphi_* ,w}$, respectively. The spaces $\mathcal{M}_{\varphi ,w}$ have been recently introduced in the paper \cite{KR12}. Given an Orlicz function $\varphi $ and a weight $w$ let
\begin{equation*}
\mathcal{M}_{\varphi ,w}=\{ f\in L^{0}:\exists _{\lambda >0}\
P_{\varphi ,w}\left( f/\lambda \right) <\infty \} ,
\end{equation*}%
where the modular $P_{\varphi ,w}$ is defined as
\begin{equation*}
P_{\varphi ,w}\left( f\right) =\inf \Big\{ \int_{I}\varphi \Big( \frac{%
f^{\ast }}{|g|}\Big) |g|:g\prec w\Big\} =\inf \Big\{ \Big\Vert \varphi
\Big( \frac{f^{\ast }}{\vert g\vert }\Big) g\Big\Vert
_1:g\prec w\Big\} .
\end{equation*}%
Here and further in the paper by $\|\cdot\|_1$ we denote the norm in the space $L^1$ of integrable functions on $I$.
In order to avoid any ambiguity in the definition of the modular $P_{\varphi
,w}$ let us agree on the convention that for any measurable functions $%
f,g\geq 0$ on $I$, if $g(t)=0$ then
\begin{equation*}
\varphi \Big( \frac{f(t)}{g(t)}\Big) g(t)=%
\begin{cases}
0\ \ \text{if}\ \ f(t)=0; \\
\infty \ \ \text{if}\ \ f(t)\neq 0.%
\end{cases}%
\end{equation*}%
It is also worth to observe that
\begin{equation}
P_{\varphi ,w}(f)=\inf \Big\{ \Big\Vert \varphi \Big( \frac{f^{\ast }}{|g|%
}\Big) g\Big\Vert _1:\Vert g\Vert _{W}=1\Big\} .  \label{eq:P}
\end{equation}%
In fact by convexity of $\varphi $ one has $\frac{1}{a}\varphi \left(
at\right) \leq \varphi \left( t\right) $ for each $t>0$ and $0<a\leq 1$.
Therefore, if $\left\Vert g\right\Vert _{M_{W}}=a<1$ then
\begin{equation*}
\Big\Vert \varphi \Big( \frac{af^{\ast }}{k|g|}\Big) \frac{kg}{a}%
\Big\Vert _1\leq \Big\Vert \varphi \Big( \frac{f^{\ast }}{%
k\vert g\vert }\Big) kg\Big\Vert _1.
\end{equation*}%
We introduce two equivalent norms on $\mathcal{M}_{\varphi ,w}$. The first
one is of the Luxemburg type,
\begin{equation*}
\left\Vert f\right\Vert _{\mathcal{M}}=\left\Vert f\right\Vert _{\mathcal{M}%
_{\varphi ,w}}=\inf \{ \lambda >0:P_{\varphi ,w}\left( f/\lambda
\right) \leq 1\} ,
\end{equation*}%
and the second one is of the Amemiya type,
\begin{equation*}
\left\Vert f\right\Vert _{\mathcal{M}}^{0}=\left\Vert f\right\Vert _{%
\mathcal{M}_{\varphi ,w}}^{0}=\inf_{k>0}\frac{1}{k}(P_{\varphi ,w}\left(
kf\right) +1) .
\end{equation*}%
By $\mathcal{M}_{\varphi ,w}$ denote the space equipped with the norm $\Vert
\cdot \Vert _{\mathcal{M}}$ and by $\mathcal{M}_{\varphi ,w}^{0}$ the space
endowed with the norm $\Vert \cdot \Vert _{\mathcal{M}}^{0}$.
The first result expresses the spaces $\mathcal{M}_{\varphi,w}$ and $\mathcal{M%
}^0_{\varphi,w}$ as Calder\'on-Lozanovskii spaces relative to the couple $(M_W,L_1)$.

\begin{proposition}
\label{prop:1} Let $\varphi $ be an $N$-function. Then
\begin{align*}
\rho_{\varphi } \left( M_{W},L^{1}\right) & =\mathcal{M}_{\varphi ,w}\ \ \ \text{and}\
\ \ \rho_{\varphi } ^{0}\left( M_{W},L^{1}\right) =\mathcal{M}_{\varphi ,w}^{0}, \\
\hat{\rho}_{\varphi }\left( L^{1},M_{W}\right) & =\mathcal{M}_{\varphi _{\ast },w}\ \ \
\text{and}\ \ \ \hat{\rho}_{\varphi }^{0}\left( L^{1},M_{W}\right) =\mathcal{M}%
_{\varphi _{\ast },w}^{0},
\end{align*}%
with equalities of corresponding norms. %\begin{equation*}
%\|f\|_{\rho_{\infty }\left( M_{W},L^{1}\right)} = \|f\|_{\mathcal{M} _{\varphi ,w}} \ \ \
%\text{and} \ \ \ \|f\|_{\rho_1\left( M_{W},L^{1}\right)} = \|f\|_{\mathcal{M}^0 _{\varphi ,w}}
%\end{equation*}
%for every $f$.
\end{proposition}

\begin{proof}
Let $f\in \rho_{\varphi } (M_{W},L^{1})$. Then
\begin{eqnarray*}
\left\Vert f\right\Vert _{\rho_{\varphi } \left( M_{W},L^{1}\right) } &=&\left\Vert
f^{\ast }\right\Vert _{\rho_{\varphi } \left( M_{W},L^{1}\right) }
=\inf \left\{ \max \{ \left\Vert g\right\Vert _{M_{W}},\left\Vert
h\right\Vert _1\} :f^{\ast }=\rho _{\varphi }\left(
|g|,|h|\right)\right\} \\
&=&\inf \Big\{ C>0:f^{\ast }=C\rho _{\varphi }\left( |g|,|h|\right)
,\left\Vert g\right\Vert _{M_{W}}\leq 1,\left\Vert h\right\Vert _1\leq
1\Big\} \\
&=&\inf \Big\{ C>0:\varphi \Big( \frac{f^{\ast }}{C\left\vert g\right\vert
}\Big) \left\vert g\right\vert =\left\vert h\right\vert ,\left\Vert
g\right\Vert _{M_{W}}\leq 1,\left\Vert h\right\Vert _1\leq 1\Big\} \\
&=&\inf \Big\{ C>0:\Big\Vert \varphi \Big( \frac{f^{\ast }}{C\left\vert
g\right\vert }\Big) g\Big\Vert _1\leq 1\text{, }\left\Vert
g\right\Vert _{M_{W}}\leq 1\Big\} \\
%&=&\inf \left\{ C>0:\left\Vert \varphi \left( \frac{f^{\ast }}{C\left\vert
%g\right\vert }\right) g\right\Vert _1\leq 1\text{, }g\prec w\right\} \\
&=&\inf \Big\{ C>0:\inf \Big\{ \Big\Vert \varphi \Big( \frac{f^{\ast }}{%
C\left\vert g\right\vert}\Big) g\Big\Vert_1:g\prec w\Big\}
\leq 1\Big\}
=\Vert f\Vert _{\mathcal{M}_{{\varphi },w}}.
\end{eqnarray*}%
%
%
%
%
%
%
%
%To explain equality (1) denote the set of first (i.e. above) inf as $Z_{1}$
%and the set of the second as $Z_{2}$. If $C\in Z_{1}$ there is $g\prec
%\omega $ such that $\left\Vert \varphi\left( \frac{f^{\ast }}{%
%C\left\vert g\right\vert }\right) \left\vert g\right\vert \right\Vert
%_1\leq 1$, then of course $\inf \left\{ \left\Vert \varphi\left( \frac{f^{\ast }}{C\left\vert g\right\vert }\right) \left\vert
%g\right\vert \right\Vert _1:g\prec w\right\} \leq 1$ and so $C\in
%Z_{2} $. If $C\in Z_{2}$ then
%\begin{equation*}
%\inf \left\{ \left\Vert \varphi\left( \frac{f^{\ast }}{C\left\vert
%g\right\vert }\right) \left\vert g\right\vert \right\Vert _1:g\prec
%w\right\} \leq 1
%\end{equation*}
%and so for each $\varepsilon >0$ there is $g\prec w$ such that $\left\Vert
%\varphi\left( \frac{f^{\ast }}{C\left\vert g\right\vert }\right)
%\left\vert g\right\vert \right\Vert _1\leq 1+\varepsilon $. But then
%\begin{equation*}
%\left\Vert \varphi\left( \frac{f^{\ast }}{\left( 1+\varepsilon
%\right) C\left\vert g\right\vert }\right) \left\vert g\right\vert
%\right\Vert _1\leq 1
%\end{equation*}%
%and so $\left( 1+\varepsilon \right) C\in Z_{1}$. Since $\varepsilon >0$ was
%arbitrary we have $\inf Z_{1}=\inf Z_{1}$.

Applying (\ref{eq:P}) we also get the second of the first two equalities
\begin{eqnarray*}
\left\Vert f\right\Vert ^{0} _{\rho_{\varphi } \left( M_{W},L^{1}\right) } &=&\left\Vert
f^{\ast }\right\Vert ^{0}_{\rho_{\varphi } \left( M_{W},L^{1}\right) }
=\inf \left\{ \left\Vert g\right\Vert _{M_{W}}+\left\Vert h\right\Vert
_1:f^{\ast }=\rho _{\varphi }\left( \left\vert g\right\vert
,\left\vert h\right\vert \right) \right\} \\
&=&\inf \Big\{ \left\Vert g\right\Vert _{M_{W}}+\Big\Vert \varphi \Big(
\frac{f^{\ast }}{\left\vert g\right\vert }\Big) g\Big\Vert _1:g\in
M_{W}\Big\} \\
&=&\inf_{k>0}\Big\{ \inf \Big\{ k+\Big\Vert \varphi \Big( \frac{f^{\ast }%
}{k\left\vert g\right\vert }\Big) kg\Big\Vert _1:\left\Vert
g\right\Vert _{M_{W}}=1\Big\} \Big\} \\
&=&\inf_{k>0}\Big\{ \inf \Big\{ k+\Big\Vert \varphi \Big( \frac{f^{\ast }%
}{k\left\vert g\right\vert }\Big) kg\Big\Vert _1:\left\Vert
g\right\Vert _{M_{W}}\leq 1\Big\} \Big\} \\
%&=&\inf_{k>0}\left\{ k+k\inf \left\{ \left\Vert \varphi \left( \frac{f^{\ast
%}}{k\left\vert g\right\vert }\right) g\right\Vert _1:g\prec w\right\}
%\right\} \\
&=&\inf_{k>0}\frac{1}{k}\left( P_{\varphi ,w}\left( kf\right) +1\right)
=\Vert f\Vert _{\mathcal{M}_{\varphi ,w}}^{0}.
\end{eqnarray*}%
The remaining equalities are proved
analogously by (\ref{repr of conj}).
\end{proof}

Now we are ready to state an isometric characterization of the (K\"othe)
dual spaces of Orlicz-Lorentz spaces.

\begin{theorem}
\label{th:01} Let $w$ be a decreasing weight and $\varphi $ be an $N$-function. Then the following holds true.\\
\par $(1)$ The Köthe dual spaces to Orlicz-Lorentz spaces $\Lambda _{\varphi ,w}$
and $\Lambda _{\varphi ,w}^{0}$ are expressed as
\begin{equation*}
\left( \Lambda _{\varphi ,w}\right) ^{\prime }=\mathcal{M}_{\varphi _{\ast
},w}^{0}\ \ \ \text{and}\ \ \ \left( \Lambda _{\varphi ,w}^{0}\right)
^{\prime }=\mathcal{M}_{\varphi _{\ast },w},
\end{equation*}%
with equality of corresponding norms. \\
\par $(2)$ Let $\varphi $ satisfy the appropriate $\Delta _{2}$-condition,   that is {\rm{(i)}} for large arguments if $I=[0,\alpha)$ with $\alpha<\infty$, or $\alpha = \infty$ and $W(\infty)<\infty$; {\rm{(ii)}} for all arguments if $I=[0,\infty )$ and $W(\infty )=\infty $. Then the dual spaces $({\Lambda _{\varphi ,w}})^{\ast }$ and $({\Lambda _{\varphi ,w}^{0}})^{\ast
} $ are isometrically isomorphic to their corresponding Köthe dual spaces.
In fact for any functional $\Phi \in (\Lambda _{\varphi ,w})^{\ast }$ $($resp., $\Phi \in (\Lambda _{\varphi ,w}^{0})^{\ast } $ $)$ there exists $
\phi \in \mathcal{M}_{\varphi _{\ast },w}^{0}$ $($resp., $\phi \in \mathcal{M
}_{\varphi _{\ast },w}$ $)$ such that
\begin{equation*}
\Phi (f)=\int_{I}f\phi ,\ \ \ f\in \Lambda _{\varphi ,w},
\end{equation*}%
and $\Vert \Phi \Vert _{(\Lambda _{\varphi ,w})^{\ast }}=\Vert \phi \Vert ^{0}_{%
\mathcal{M}_{\varphi _{\ast },w}}$ $($resp., $\Vert \Phi \Vert
_{(\Lambda _{\varphi ,w}^{0})^{\ast }}=\Vert \phi \Vert _{\mathcal{M}%
_{\varphi _{\ast },w}}$$)$.
\end{theorem}

\begin{proof}
By Lozanovskii's representation theorem \cite{Lo78b} for any Banach function
spaces $E,F$ we have
\begin{equation*}
( \rho ( E,F) ) ^{\prime }=\hat{\rho}^{0}(
E^{\prime },F^{\prime }) \ \ \ \text{and}\ \ \ ( \rho ^{0}(
E,F) ) ^{\prime }=\hat{\rho}( E^{\prime },F^{\prime
}) ,
\end{equation*}%
with norm equalities. Notice that $\left( \Lambda _{w}\right) ^{\prime
}=M_{W}$. This fact was proved in \cite[Theorem 5.2, p. 112]{KPS} under the assumption that $W(\infty )=\infty $ in case of $I=[0,\infty)$, however the same proof works for any decreasing locally integrable weight function $w$. Thus by (\ref{repr of conj}), (\ref{eq:L}) and Proposition \ref{prop:1} we get
\begin{equation*}
( \Lambda _{\varphi ,w}) ^{\prime }=( \rho ( L^{\infty
},\Lambda _{w}) ) ^{\prime }=\hat{\rho}^{0}(
L^{1},M_{W}) =\mathcal{M}_{\varphi _{\ast },w}^{0}.
\end{equation*}%
The second equality can be shown analogously. The second part of the
hypothesis follows from the well known fact that the Orlicz-Lorentz spaces
are order continuous  \cite{Kam90} under the assumption of the appropriate $\Delta _{2}$-condition,
and the general theorem stating that the Köthe dual space $E^{\prime }$ of an order continuous Banach
function space $E$ is isometrically isomorphic via integral functionals to
the dual space $E^{\ast }$ \cite[Theorem 4.1, p. 20]{BS88}.
\end{proof}

\section{An algorithm for computing $P_{\protect\varphi ,\protect w %
}(f)$ for a decreasing simple function $f$}

%\bigskip \textbf{Zamienic Fact na Lemma - KL: }Chyba nie potrzebujemy tego
%faktu w tej formie dalej. Mo\.{z}na go oczywi\'{s}cie zostawic, ale mo\.{z}e
%po Lemacie 1 albo Lemat 1 podzieli\'{s}\'{c} na a i b? Kiedy\'{s} pr\'{o}bowa%
%\l em udowodni\'{c}, \.{z}e dla dowolnych r.i. $E,F$ w definicji normy $%
%\left\Vert f^{\ast }\right\Vert _{\rho _{p}\left( E,F\right) }$ mo\.{z}na bra%
%\'{c} $g=g^{\ast }$,$h=h^{\ast }$, ale da\l em rad\k{e} pokaza\'{c} to
%tylko,\ co do r\'{o}wnowa\.{z}no\'{s}ci norm, t.j. mniej wi\k{e}cej tak
%samo, jak w \cite{HuKaMa02}. W szczeg\'{o}lno\'{s}ci oznacza\l oby to, \.{z}%
%e w poni\.{z}szym fakcie mo\.{z}na bra\'{c} nie tylko proste, ale te\.{z}
%nierosn\k{a}ce $f,g$. Co ciekawe, Algorytm 1 pokazuje, \.{z}e w przypadku $%
%\rho _{\varphi _{\ast },\infty }\left( M_{id/W},L^{1}\right) $ elementy $%
%g^{f}$ i $h=\varphi \left( \frac{f^{\ast }}{g^{f}}\right) g^{f}$, kt\'{o}re
%realizuj\k{a} norm\k{e} $f^{\ast }$ s\k{a} w\l a\'{s}nie nierosn\k{a}ce, bo $%
%\frac{f^{\ast }}{g^{f}}$ jest nierosn\k{a}cy!!!

In this section our goal is to find a function $g$ which minimizes the
formula $P_{\varphi,w}(f)$ for a given simple decreasing function $f=\Sigma _{i=1}^{n}a_{i}\chi _{[
t_{i-1},t_{i}) }$ with $a_1 > a_2 > \dots > a_n > 0$ and $0=t_0 < t_1
<\dots <t_n < \infty$. This process consists of several steps and leads to an algorithm which reveals that such a function $g$ exists and depends only on $f$ and $w$, but not on $\varphi$.

First in Lemma \ref{lem:2} we show that the minimizing function $g$ has to be also
simple and decreasing. In the second step in Lemma \ref{lem:4} we show  that such
a minimizing function $g$ exists. Next, in Lemma \ref{lem:5} it is proved that $G\left(
t_{n}\right) =W\left( t_{n}\right) $ and then in Theorem \ref{th:1} is demonstrated
that
%\begin{equation*}
$g=\sum_{j=0}^{m-1}\lambda _{j}f\chi _{[ t_{i_{j}},t_{i_{j+1}}) }$
%\end{equation*}%
for some $(\lambda _{i})_{j=0}^{m-1}$, $(t_{i_{j}})_{j=0}^{m-1}$ and $m\leq
n $, where $W(t_{i_j})=G(t_{i_j})$. This shows that $g$ needs to be piecewise
proportional to $f$ and the ratios $%
\lambda _{j}$ are determined by the points $t_{i_{j}}$. Therefore in order
to find $g $ it is sufficient to determine points $t_{i_{j}}$. This process
will be described by Algorithm A. Applying finally Theorem \ref{th:1} and %
Lemma \ref{th:2}, we finish with proving that Algorithm A produces the function $g$
that minimizes $P_{\varphi ,w}(f)$.

\begin{lemma}
\label{lem:2} If $f=\Sigma _{i=1}^{n}a_{i}\chi _{A_{i}}$ where $a_{1}>\dots
>a_{n}>0$ and $A_{i}=\left[ t_{i-1},t_{i} \right) $ with $%
0=t_{0}<t_{1}<...<t_{n}<\infty $, then
\begin{equation}
P_{\varphi ,w }\left( f\right) =\inf \left\{
\begin{array}{c}
\Vert \varphi ( \frac{f}{g}) g\Vert _1:g\prec
w ,\text{ } \\
\text{and }g=\Sigma _{i=1}^{n}b_{i}\chi _{A_{i}}\text{ with }b_{1}\geq
b_{2}\geq ...\geq b_{n}>0%
\end{array}
\right\} .  \label{mod dla prost}
\end{equation}
\end{lemma}

\begin{proof}
Let $f=\Sigma _{i=1}^{n}a_{i}\chi _{A_{i}}$ satisfy the assumptions.
Corollary 4.5 in \cite{KR12} states that
\begin{equation*}
P_{\varphi ,w }( f) =\inf \{ \Vert \varphi (
f/g) g\Vert _1:g\prec w,\ 0\leq g\downarrow
\} ,  \label{df mod z gw}
\end{equation*}
where $g\downarrow$ means that $g$ is decreasing.
Fix some $g\prec w,$ $g\downarrow $ and put
\begin{equation*}
h=\varphi \Big( \frac{f}{g}\Big) g\text{ and }\tilde{h}=\varphi \Big(
\frac{f}{Tg}\Big) Tg,
\end{equation*}%
where%
\begin{equation*}
T:g\mapsto \sum_{i=1}^{n}\Big(\frac{1}{|A_{i}|} \int_{A_{i}}g\Big)\chi
_{A_{i}} .
\end{equation*}%
Since $g$ is decreasing $Tg$ is also decreasing. Therefore it is enough to
show that $\Vert \tilde{h}\Vert _1\leq \Vert
h\Vert _1~$\ and $Tg\prec w$.

%By Jensen's inequality for concave
%functions we get
%\begin{equation*}
%\varphi ^{-1}\left( \frac{\bar{h}}{Tg}\right) Tg=f=\varphi ^{-1}\left( \frac{%
%h}{g}\right) g\leq \varphi ^{-1}\left( \frac{Th}{Tg}\right) Tg,
%\end{equation*}%
%and since $\varphi ^{-1}$ is increasing it must be $\bar{h}\leq Th$. On the
%other hand $\left\Vert Th\right\Vert _1\leq \left\Vert h\right\Vert
%_1$ and therefore $\left\Vert \bar{h}\right\Vert _1\leq
%\left\Vert h\right\Vert _1$. Relation $Tg\prec \omega $ holds since $%
%Tg\prec g\prec w$.

By  Proposition 3.7 in \cite[Chap. 2]{BS88} we have $Tg\prec g$ and so $Tg \prec w$. By
convexity of the function $s\mapsto \varphi \left( a/s\right) s$, $%
a>0$, and Jensen's inequality for convex functions we have for every $%
i=1,\dots,n$,
\begin{equation*}
\varphi \Big( \frac{a_{i}}{(1/|A_{i}|)} \int_{A_{i}}g \Big) \frac{1%
}{|A_{i}|} \int_{A_{i}}g \leq \frac{1}{|A_{i}|}\int_{A_{i}}\varphi \Big(
\frac{a_{i}}{g}\Big) g,
\end{equation*}%
which gives

\begin{equation*}
\Vert \tilde{h}\Vert _1=\sum_{i=1}^{n}|A_{i}|\varphi \Big(
\frac{a_{i}}{({1}/{|A_{i}|})}\int_{A_{i}}g\Big) \frac{1}{|A_{i}|}%
\int_{A_{i}}g\leq \sum_{i=1}^{n}\int_{A_{i}}\varphi \Big( \frac{a_{i}}{g}%
\Big) g=\Vert h\Vert _1 ,
\end{equation*}%
and the proof is finished.
\end{proof}

\begin{lemma}
\label{lem:3} Suppose that $g=g^{\ast }=\Sigma _{i=1}^{n}b_{i}\chi _{\left[
t_{i-1},t_{i}\right) }$ where $0=t_{0}<t_{1}<...<t_{n}<\infty $. Then
\begin{equation*}
\inf_{0<t\leq t_{n}}\frac{W\left( t\right) }{G\left( t\right) }%
=\min_{i=1,\dots ,n}\frac{W\left( t_{i}\right) }{G\left( t_{i}\right)}.
%\text{.}
\end{equation*}%
In particular $g\prec w$ if and only if $G\left( t_{i}\right) \leq W\left(
t_{i}\right) $ for each $i=1,\dots ,n$.
\end{lemma}

\begin{proof}
The left-hand side is clearly majorized by the right-hand one. Conversely if for some $\theta\ge 0$,
\[
W(t_i)\ge \theta\, G(t_i),\ \ i=1,\dots,n,
\]
which remains trivially true for $i=0$, then by concavity of $W$ and the fact that $G$ is affine on each segment $[t_i,t_{i+1}]$, we have for every $\lambda\in[0,1]$ and $i=0,\dots, n-1$,
\begin{align*}
W((1-\lambda)t_i+\lambda t_{i+1})&\ge (1-\lambda)W(t_i)+\lambda W(t_{i+1})\\
&\ge (1-\lambda)\theta\, G(t_i)+\lambda\, \theta G(t_{i+1})=\theta\, G((1-\lambda)t_i+\lambda t_{i+1}),
\end{align*}
which proves the converse inequality.
\end{proof}

\begin{lemma}
\label{lem:4} Let $\varphi $ be an $N$-function. Then for a simple function $%
f=\Sigma _{i=1}^{n}a_{i}\chi _{A_{i}}$ such that $a_{1}>\dots >a_{n}>0$ and $%
A_{i}=\left[ t_{i-1},t_{i}\right) $ where $0=t_{0}<t_{1}<...<t_{n}<\infty $,
there exists $g=\sum_{i=1}^{n}b_{i}\chi _{A_{i}}$ with $b_{1}\geq \dots \geq
b_{n}>0$, $\left\Vert g\right\Vert _{M_{W}}=1$ and such that $P_{\varphi
,w}\left( f\right) =\Vert \varphi ( f/g)
g\Vert _1$. Consequently in the definition of $P_{\varphi
,w}\left( f\right) $ the infimum is attained.
\end{lemma}

\proof
In fact, by Lemma \ref{lem:2} the infimum in the definition of $P_{\varphi ,w}\left( f\right) $ may
be considered over $g=\sum_{i=1}^{n}b_{i}\chi _{A_{i}}$ with $
b_{1}\geq \dots \geq b_{n}>0$ such that $g\prec w$. Applying now Lemma \ref{lem:3} the
condition $g\prec w$ is equivalent to $G(t_{k})=\sum_{i=1}^{k}b_{i}\left%
\vert A_{i}\right\vert \leq W\left( t_{k}\right) $ for each $k=1,\dots ,n$.
But those constrains define the set
\begin{equation*}
C=\Big\{b=(b_{1},\dots ,b_{n}):b_{1}\geq \dots \geq b_{n}>0,\
\sum_{i=1}^{k}b_{i}|A_{i}|\leq W(t_{k}),k=1,\dots ,n\Big\},
\end{equation*}%
which is relatively compact in $\mathbb{R}^{n}$. Hence %the infimum in formula (%
%\ref({mod dla prost}) is attained at some $b\in \bar{C}$, where $\bar{C}$ is the
%closure of $C$. If therefore
if the sequence $\left( b^{k}\right) =\left(
\left( b_{1}^{k},...,b_{n}^{k}\right) \right) _{k=1}^{\infty }\subset
\mathbb{R}^{n}$ is minimizing for the infimum in formula (\ref{mod dla prost})
in the definition of $P_{\varphi ,w}\left( f\right) $,  there is a subsequence $(b^{k_{j}})$ such that
$b^{k_{j}}\rightarrow b\in \bar{C}$. Let us
show that $b\in C$. If for the contrary $b\in \bar{C}\backslash C$ then
$b_{i}=0$ for some $i=1,\dots ,n$, which means that $b_{i}^{k_{j}}%
\rightarrow 0$ if $j\rightarrow \infty $. Setting $g_{k_{j}}=%
\sum_{i=1}^{n}b_{i}^{k_{j}}\chi _{A_{i}}$, by the fact that $\varphi$ is an $N$-function we
get for $j\rightarrow \infty $,
\begin{equation*}
\Big\Vert \varphi \Big( \frac{f}{g_{k_{j}}}\Big) g_{k_{j}}\Big\Vert
_1\geq \varphi \Big( \frac{a_{i}}{b_{i}^{k_{j}}}\Big)
b_{i}^{k_{j}}\left\vert A_{i}\right\vert \rightarrow \infty ,
%\label{realizing g}
\end{equation*}%
which ensures that $\left( b^{k_{j}}\right) $ cannot be minimizing for  the
infimum in the definition of $P_{\varphi ,w}\left( f\right) $,  a contradiction.
\endproof

%\bigskip

%\textbf{Remark.} Notice that the above lemma is also a consequence of the general Proposition 1 from \cite{Re88}.

%\bigskip

%(KL: in the above claim it is necessary that $\varphi $ is N - function, but
%probably it is able to prove the same using just monotonicity (or strictly
%monotonicity) of $\varphi \left( \frac{a}{t}\right) t$? Then we would avoid
%this assumption (i.e. N-function) quite easily, since everywhere below it is
%necessary only (!?) because of the above remark. Then if we knew also (\ref%
%{repr of conj}) we would have the results for all Orlicz functions (at most
%the case $\varphi \not<\infty $ would need to be checked more carefully
%because of Remark 4))\newline

\textbf{Example.} We present an example which shows that for decreasing
simple functions $f$ the functions $g$ that minimize $P_{\varphi,w}(f)$
depend on $f$.

Let $\varphi \left( t\right) =t^{2}$, $w\left( t\right) =1/2\sqrt{t}$, $t>0$.
Define the family of functions $f_{x}:=x\chi _{\left( 0,1\right) }+1\chi _{\left( 1,4\right) }$ on $(0,\infty )$ for $x\geq 1$.
Then by Lemmas \ref{lem:2} -- \ref{lem:5},
\begin{equation*}
P_{\varphi ,w}\left( f_{x}\right) =\min \left\{
\begin{array}{c}
\frac{x^{2}}{b_{1}}+\frac{3}{b_{2}}:g=b_{1}\chi _{\left[ 0,1\right)
}+b_{2}\chi _{\left[ 1,4\right) },\text{ } \\
\text{with }1\geq b_{1}\geq b_{2}\text{ and }b_{1}+3b_{2}=2%
\end{array}%
\right\} .
\end{equation*}%
Applying Lagrange multipliers method to minimize the function $\psi \left(
b_{1},b_{2}\right) =\frac{x^{2}}{b_{1}}+\frac{3}{b_{2}}$ with constraint $%
b_{1}+3b_{2}=2$ gives the solution
$b_{1}=\frac{2x}{x+3}\ ,\ \ b_{2}=\frac{2}{x+3}$.
We have $b_{1}\geq b_{2}$ for $x\geq 1$. Moreover, if $1\leq x\leq 3$ then $%
b_{1}\leq 1$. If $x\geq 3$ then there is no extremum in the set defined by
constraints~$1\geq b_{1}\geq b_{2}$ and $b_{1}+3b_{2}=2$ and therefore $\psi
$ attains its minimum at $\left( 1,1/3\right) $ or $\left( 1/2,1/2\right) $.
Finally we get
\[
P_{\varphi ,w}\left( f_{x}\right) =
\begin{cases}
&\frac{x\left( x+3\right) }{2}+\frac{%
3\left( x+3\right) }{2}\ \ \text{ for }1\leq x\leq 3, \\
&x^{2}+9\ \ \text{ for }x\geq 3,
\end{cases}
\]
and it is clear that  $g$ cannot be chosen independently of $f_{x}$.

\begin{lemma}
\label{lem:5} Let $\varphi $ be an $N$-function. Let $f=\Sigma
_{i=1}^{n}a_{i}\chi _{A_{i}}$ for some $a_{i}$ with $a_{1}> \dots>a_{n}>0$
and $A_{i}=\left[ t_{i-1},t_{i}\right) $ where $0=t_{0}<t_{1}<\dots<t_{n}<%
\infty $. If $g=\Sigma _{i=1}^{n}b_{i}\chi _{A_{i}}=g^{\ast }$ is a
minimizing function for $P_{\varphi ,w}\left( f\right) $ then
\begin{equation}
G\left( t_{n}\right)=\int_{0}^{t_{n}} g=\int_{0}^{t_{n}}w =W(
t_{n}).  \label{lem 2}
\end{equation}
\end{lemma}

\begin{proof}
One may assume that\textbf{\ }$\left\Vert g\right\Vert _{M_{{W}}}=1$. By
Lemma \ref{lem:4} we also have that $b_1 \ge\dots\ge b_n >0$, so $g>0$ on $%
[0,t_n)$. Suppose\textbf{\ }$g$ does not satisfy (\ref{lem 2}) that is $%
G(t_n) < W(t_n)$.
We will then find a new function $h$ such that $h\prec w$
and $\Vert \varphi (f/h) h\Vert
_1<\Vert \varphi ( f/g) g\Vert _1$
contradicting the minimality of $g$.

By Lemma \ref{lem:3}, $\left\Vert g\right\Vert _{M_{{W}}}=1$ is equivalent
to
\begin{equation*}
1=\sup_{t_{n}\geq t>0}\frac{G\left( t\right) }{W\left( t\right) }%
=\inf_{t_{n}\geq t>0}\frac{W\left( t\right) }{G\left( t\right) }%
=\min_{i=1,\dots,n}\frac{W\left( t_{i}\right) }{G\left( t_{i}\right) } .
\end{equation*}
It follows that $\{ i>0:W( t_{i}) = G(t_{i}) \}\not=\emptyset $. Let
\begin{equation*}
i_1=\max \{ i>0:W( t_{i}) = G(t_{i}) \} \ \ \ \text{and}\ \ \ \
\gamma _{1}=\min_{i_{1}<i\leq n}\Big\{ \frac{W\left( t_{i}\right) -W\left(
t_{i_{1}}\right) }{G\left( t_{i}\right) -G\left( t_{i_{1}}\right) }\Big\}.
\end{equation*}
Since $G(t_n) < W(t_n)$  we have $i_1 < n$. Then $W(t_i) > G(t_i)$ for all $i>i_1$, and thus it is clear that $\gamma_1 > 1$. Set
\[
g_{1}=g\chi _{[ 0,t_{i_{1}}) }+\gamma _{1}g\chi _{[
t_{i_{1}},t_{n})}.
\]
Note that $g_{1}=g_{1}^{\ast }$. In fact, since $g$ is decreasing, it is
sufficient to show that $b_{i_{1}}\geq \gamma _{1}b_{i_{1}+1}$. We have $%
G( t_{i_{1}}) =G_{1}( t_{i_{1}}) =W(
t_{i_{1}}) $ and $G_{1}( t_{i_{1}+1}) \leq W(
t_{i_{1}+1}) $ by definition of $\gamma _{1}$. Then
\begin{equation*}
w( t_{i_{1}}) ( t_{i_{1}+1}-t_{i_{1}}) \geq W(
t_{i_{1}+1}) -W( t_{i_{1}}) \geq G_{1}(
t_{i_{1}+1}) -G_{1}( t_{i_{1}}) =\gamma
_{1}b_{i_{1}+1}( t_{i_{1}+1}-t_{i_{1}}) .
\end{equation*}%
On the other hand $G( t_{i_{1}-1}) =G_{1}(
t_{i_{1}-1}) \leq W( t_{i_{1}-1}) $, so
\begin{equation*}
b_{i_{1}}( t_{i_{1}}-t_{i_{1}-1}) =G_{1}( t_{i_{1}})
-G_{1}( t_{i_{1}-1}) \geq W( t_{i_{1}}) -W(
t_{i_{1}-1}) \geq w( t_{i_{1}}) (
t_{i_{1}}-t_{i_{1}-1}) .
\end{equation*}%
Therefore
$
b_{i_{1}}\geq w( t_{i_{1}}) \geq \gamma _{1}b_{i_{1}+1}.
$
We also have that $g_{1}\prec w.$ Indeed in view of Lemma \ref{lem:3} and
definition of $i_{1}$ it is enough to check the inequality $G_{1}(
t_{i}) \leq W( t_{i}) $ for $i>i_{1}$. We have
\begin{eqnarray*}
G_{1}( t_{i}) &=&G( t_{i_{1}}) +\gamma _{1}(G(
t_{i}) -G( t_{i_{1}}) ) \\
&\leq& G( t_{i_{1}}) +\frac{W( t_{i}) -W(
t_{i_{1}}) }{G( t_{i}) -G( t_{i_{1}}) }(G(
t_{i}) -G( t_{i_{1}}) )=W( t_{i}).
\end{eqnarray*}%
However, $\Vert \varphi
( f/g_{1}) g_{1}\Vert _1<\Vert \varphi
( f/g) g\Vert _1$ in view of $g_{1}\geq g$
and $g_{1}\not=g$ and the fact that $\varphi ( a/t) t$ is a
strictly decreasing function of $t$ for each $a>0$ (by the assumption that $%
\varphi $ is an $N$-function).
This contradicts the fact that $g$ realizes the infimum in the definition of $P_{\varphi,w}(f)$.
\end{proof}

%(KL: It is worth to see the above algorithm on the picture of $W$
%and $G$. I shall make some pictures later on)

\begin{lemma}
\label{Decomposition} Let $\varphi$ be an $N$-function. Let $%
f=\Sigma_{i=1}^{n}a_{i}\chi _{A_{i}}$ for some $a_{i}$ with $a_{1}>
\dots>a_{n}>0$, $A_{i}=[ t_{i-1},t_{i}) $ where $%
0=t_{0}<t_{1}<\dots<t_{n}<\infty $, and $g=\Sigma
_{i=1}^{n}b_{i}\chi_{A_{i}}=g^{\ast }$ be a minimizing function for $%
P_{\varphi ,w}\left( f\right) $. Assume also that for some $k=1,\dots,n-1$
we have $G(t_k)=W(t_k)$. Then $g\chi_{[0,t_k)}$ is a minimizing function for $
P_{\varphi ,w}( f\chi_{[0,t_k)}) $ that is
\begin{equation*}
\|\varphi(f/g) g \chi_{[0,t_k)}\|_1 =
P_{\varphi,w}(f\chi_{[0,t_k)}),
\end{equation*}
while
\begin{equation}  \label{eq:00}
\|\varphi(f/g) g \chi_{[t_k,t_n)}\|_1 =
\inf \Big\{%
\begin{array}{c}
\sum_{i={k+1}}^n \varphi(a_i/c_i) c_i |A_i|: c_{k+1} >
\dots > c_n > 0, \\
\sum_{i=k+1}^j c_i |A_i| \le W(t_j) - W(t_k), \ k<j \le n%
\end{array}
\Big\}.
\end{equation}
\end{lemma}

\begin{proof}
By Lemma \ref{lem:4} we have that $b_1 \ge \dots \ge b_n > 0$ and by
$g\prec w$ and $G(t_k) = W(t_k)$ it holds
\begin{equation*}
\sum_{i=k+1}^j b_i|A_i|\le W(t_j)-W\left(t_k\right),\quad k<j\le n .  \label{constraints}
\end{equation*}
Note that the problem of minimizing (\ref{eq:00}) is equivalent to
minimizing $P_{\varphi ,w_{k}}\left( f_{k}\right)$,
where for $t\in I$ we let
\begin{equation*}
f_{k}( t)=(f\chi _{[ t_k,t_n) })(t+t_{k}),\ \
\ w_{k}( t)=(w\chi _{[ t_k,t_n) })(t+t_{k}).
\end{equation*}
By Lemma \ref{lem:4} applied to the interval $[0, t_n-t_k)$ there is a
simple, decreasing function $h^{(1)}= \sum_{i=k+1}^n h_{i}\chi
_{A_{i}-t_k}\prec w_k$ that minimizes $P_{\varphi ,w_{k}}\left( f_{k}\right)
$. On the other hand by Lemma \ref{lem:4} and Lemma \ref{lem:5} applied to
the interval $[0, t_k)$ there is a simple, decreasing function $
h^{(2)}=\sum_{i=1}^kh_{i}\chi _{A_{i}}\prec w\chi_{[0,t_k)}$ that minimizes $%
P_{\varphi ,w}\left( f\chi_{[0,t_k)}\right) $, and it holds that $
H^{(2)}(t_k)= \int_0^{t_k} h^{(2)} = W(t_k)$. Thus
\begin{align}
\sum_{i=1}^k \varphi(a_i/h_i)h_i|A_i |= P_{\varphi ,w}\left(
f\chi_{[0,t_k)}\right)\le \sum_{i=1}^k \varphi(a_i/b_i)b_i|A_i | ,
\label{inequation1} \\
\sum_{i=k+1}^n \varphi(a_i/h_i)h_i|A_i |= P_{\varphi ,w_{k}}\left(
f_{k}\right)\le \sum_{i=k+1}^n \varphi(a_i/b_i)b_i|A_i | ,
\label{inequation2}
\end{align}
and so
\begin{equation}
\sum_{i=1}^n \varphi(a_i/h_i)h_i|A_i |\le \sum_{i=1}^n
\varphi(a_i/b_i)b_i|A_i | .  \label{h-better-than-g}
\end{equation}
Now let $h:=\sum_{i=1}^nh_{i}\chi _{A_{i}}$ and note that $h$  is
decreasing. Indeed we will show that $h_k\ge h_{k+1}$.
Since $H(t_{k-1}) = H^{(2)}(t_{k-1})\le W(t_{k-1})$, and $H(t_k) =
H^{(2)}(t_k)=W(t_k)$ we have
\begin{equation*}
h_k|A_k|=H(t_k)-H(t_{k-1})\ge W(t_k)-W(t_{k-1})\ge w_k |A_k| ,
\end{equation*}
while
\begin{align*}
h_{k+1}|A_{k+1}|&=H(t_{k+1})-H(t_k)= h_{k+1}|A_{k+1}| = H^{(1)}(t_{k+1}-t_k) \\
& \le W_k(t_{k+1}-t_k) =W(t_{k+1})-W(t_k)\le w_{k}|A_{k+1}| ,
\end{align*}
and thus
$
h_k\ge w_k\ge h_{k+1}
$.
Note also that $h\prec w$, since $H(t_i)\le W(t_i)$, $i=1,\dots k$, by $%
h^{(2)}\prec w$, $H(t_k)=  H^{(2)}(t_k) = W(t_k)$, and $H(t_i)-H(t_k)\le
W(t_i)-W(t_k)$, $i=k+1,\dots n$, by $h^{(1)}\prec w_k$. Then
\begin{equation*}
\sum_{i=1}^n \varphi(a_i/h_i)h_i|A_i |\ge P_{\varphi,w}(f)=\sum_{i=1}^n
\varphi(a_i/b_i)b_i|A_i | .
\end{equation*}
Consequently we have equality in (\ref{h-better-than-g}) and also in both
inequalities (\ref{inequation1}), (\ref{inequation2}), and thus the proof is
completed.
\end{proof}

\begin{lemma}\label{th:2}
Let $\varphi $ be a strictly convex $N$-function. Let $%
f=\Sigma _{i=1}^{n}a_{i}\chi _{A_{i}}$ for some $a_{i}$ with $%
a_{1}>\dots>a_{n}>0$ and $A_{i}=\left[ t_{i-1},t_{i}\right) $ where $0 =
t_{0}<t_{1}<\dots<t_{n}<\infty$ and
$g=\sum_{i=1}^{n}b_i \chi _{A_i}$ with $b_1,\dots, b_n>0$.
If $g\ne\lambda f$ for $\lambda=\frac{G(t_n)}{F(t_n)}$ then
\begin{equation*}
\Vert \varphi ( f/g) g\Vert
_1>\Vert \varphi ( f/[\lambda f]) \lambda
f\Vert _1.
\end{equation*}
\end{lemma}
\proof
Set $\lambda_i= b_i/a_i$, $i=1,\dots, n$. If all the $\lambda_i$ are equal to, say, $\lambda'$  then
$g=\lambda'\sum\limits_{i=1}^n a_i\chi_{A_i}=\lambda' f$ and $G(t_n)=\lambda' F(t_n)$ thus $\lambda' =\lambda$. If $g\not=\lambda f$, then not all $\lambda_i$ are equal and
by Jensen's inequality and strict convexity of $\varphi$ it holds
\begin{align*}
\frac{1}{G( t_{n}) }\Big\Vert \varphi \Big( \frac{f}{g}\Big)
g\Big\Vert _1&=\sum_{i=1}^n\varphi \Big( \frac{1}{\lambda _{i}}
\Big) \frac{\lambda _i a_i|A_i|}{G( t_{n}) }
 >\varphi \Big( \sum_{i=1}^{n}\frac{1}{\lambda _{i}}\frac{\lambda
_{i}a_i |A_i|}{%
G( t_{n}) }\Big) \\
&=\varphi \Big( \frac{1}{\lambda }\Big)
=\frac{1}{G( t_{n}) }\varphi \Big( \frac{1}{\lambda }\Big)
\left\Vert \lambda f\right\Vert _1 =\frac{1}{G( t_{n}) }%
\Big\Vert \varphi \Big( \frac{f}{\lambda f}\Big) \lambda f\Big\Vert
_1.
\end{align*}
\endproof

%\textbf{Remark.} Notice that in the above we don't assume $\lambda F\leq W$.

%The symbol $C^\infty$ stands as usual for the set of functions $f:\mathbb{R}
%\to \mathbb{R}$ continuously differentiable infinitely many times.

\begin{theorem}
\label{th:1} Let $\varphi$ be a strictly convex $N$-function.
Let $f=\Sigma _{i=1}^{n}a_{i}\chi_{A_{i}}$ for some $a_{i}$
with $a_{1}>\dots>a_{n}>0$ and $A_{i}=[
t_{i-1},t_{i}) $ where $0=t_{0}<t_{1}<\dots<t_{n}<\infty $. If
$g=\Sigma _{i=1}^{n}b_{i}\chi _{A_{i}}=g^{\ast }$ is
a simple function realizing the infimum in the definition of
$P_{\varphi ,w}\left( f\right)$ then $g$ has to be of
the form
\begin{equation}
g=\sum_{j=0}^{m-1}\lambda _{j}f\chi _{[ t_{i_{j}},t_{i_{j+1}}) },
\label{postac}
\end{equation}%
for some $\lambda_j$, $j= 0,1,\dots,m-1$, where $0=i_{0}<i_{1}<\dots<i_{m}=n$
and%
\begin{equation*}
G( t_{i_{j}}) = W( t_{i_{j}})
\end{equation*}%
for each $j=0,1,\dots, m$.
\end{theorem}

\begin{proof}
Let $g$ satisfy the assumptions of the theorem. By Lemma \ref{lem:4} we have
that $b_1 \ge \dots \ge b_n > 0$ and by Lemma \ref{lem:5} that $%
G\left( t_{n}\right) =W\left( t_{n}\right) $. Define a finite sequence $%
\left( i_{0}, i_1,\dots, i_m\right)$ as
\[
i_{0} = 0 \ \ \ \text{and}\ \ \
i_{j} = \min\left\{ i>i_{j-1}:G\left( t_{i}\right) =W\left( t_{i}\right)\right\}, \ \ j=1,\dots,m .
\]
As a consequence of applying $(m-1)$ times Lemma \ref{Decomposition} to $f$,
that is decomposing $f$ first as $f\chi_{[0,t_{i_{m-1}})}+f\chi_{%
[t_{i_{m-1}},t_n]}$, then $f\chi_{[0,t_{i_{m-1}})}$ as $%
f\chi_{[0,t_{i_{m-2}})}+ f\chi_{[t_{i_{m-2}},t_{i_{m-1}})}$, etc., we obtain
that if $g=g^{\ast }=\Sigma _{i=1}^{n}b_{i}\chi _{A_{i}}$ minimizes $%
P_{\varphi ,w}( f) $ and $G( t_{i_{j}}) =W(
t_{i_{j}}) $, $G( t_{i_{j+1}}) =W( t_{i_{j+1}}),$
then $( b_{i_{j}+1},...,b_{i_{j+1}}) $ also minimizes the sum
\begin{equation*}
\sum_{i=i_j+1}^{i_{j+1}}\varphi \left( a_{i}/b_{i}\right) b_{i}\left\vert
A_{i}\right\vert
\end{equation*}
under the constraints $B_j(k):=\sum_{i=i_{j}+1}^{k}b_{i}\left\vert
A_{i}\right\vert \leq W( t_{k}) -W( t_{i_{j}}) $ for $%
k= i_j+1,\dots, i_{j+1}$.

Therefore we may consider each interval $[ t_{i_{j}},t_{i_{j+1}})
$, $j = 0,1,\dots, m-1$, separately and we will show that $g\chi_{[t_{i_j},
t_{i_{j+1}})} = \lambda_j f \chi_{[t_{i_j}, t_{i_{j+1}})}$ where
\begin{equation*}
\lambda _{j}=\frac{W( t_{i_{j+1}}) -W( t_{i_{j}}) }{%
F( t_{i_{j+1}}) -F( t_{i_{j}}) } .
\end{equation*}
If $t_{i_{j+1}}=t_{i_{j}+1}$ then
\begin{align*}
g\chi _{[ t_{i_{j}},t_{i_{j+1}}) }
%\frac{1}{t_{i_{j+1}}-t_{i_{j}}%
%}\left(\int_{\left[ t_{i_{j}},t_{i_{j+1}}\right) }w \right)\chi _{\left[
%t_{i_{j}},t_{i_{j+1}}\right) }
&=\frac{F( t_{i_{j+1}}) -F( t_{i_{j}}) }{%
t_{i_{j+1}}-t_{i_{j}}}\lambda_j \chi _{[ t_{i_{j}},t_{i_{j+1}})}
=\lambda _{j}f\chi _{[ t_{i_{j}},t_{i_{j+1}}) } .
\end{align*}
If $t_{i_{j+1}}>t_{i_{j}+1}$ then for all $t_{i_{j}}<t_{i}<t_{i_{j+1}}$ one
has $G\left( t_{i}\right) <W\left( t_{i}\right) $. In this case consider the
function $\psi _{j}:\mathbb{R}_{+}^{i_{j+1}-i_{j}}\rightarrow\mathbb{R}_{+}$ defined
for $b= (b_{i_j+1}, \dots, b_{i_{j+1}})$ as
\begin{equation*}
\psi_j(b) = \psi _{j}\left( b_{i_{j}+1},...,b_{i_{j+1}}\right)
=\sum_{i=i_{j}+1}^{i_{j+1}}\varphi \left( a_{i}/b_{i}\right) b_{i}\left\vert
A_{i}\right\vert ,
\end{equation*}%
and define the set
\begin{equation*}
C_{j}=\left\{
\begin{array}{c}
b\in\mathbb{R}_{+}^{i_{j+1}-i_{j}}:b_{i_{j}+1}\ge b_{i_{j}+2}\ge \dots \ge
b_{i_{j+1}}>0, \\
\forall _{i_{j}+1\leq k<i_{j+1}}B_{j}\left( k\right) <W\left( t_{k}\right)
-W\left( t_{i_{j}}\right) , \\
B_{j}\left( i_{j+1}\right) =W\left( t_{i_{j+1}}\right) -W\left(
t_{i_{j}}\right) %
\end{array}%
\right\}.
\end{equation*}%
Notice that the condition
\begin{equation*}
B_j(k)= \sum_{i=i_{j}+1}^{k}b_{i}\vert A_{i}\vert <W(
t_{k}) -W( t_{i_{j}}) ,\ \ k=i_{j}+1,...,i_{j+1}-1,
\end{equation*}%
is a consequence of the relation $g\prec w$ and definition of $i_{j}$ and $%
i_{j+1}$. In fact, by Lemma \ref{lem:3}, $g\prec w$ is equivalent to $%
G( t_{i}) \leq W( t_{i}) $ for each $i=1,\dots,n$ and
by definition of $i_{j}$ and $i_{j+1}$ we have $G( t_{k})
<W( t_{k}) $ for each $k=i_{j}+1,...,i_{j+1}-1$. It follows that
for $k=i_{j}+1,...,i_{j+1}-1$,
\begin{equation*}
G( t_{i_{j}}) +\sum_{i=i_{j}+1}^{k}b_{i}\vert
A_{i}\vert =G( t_{k}) <W( t_{k}) =W(
t_{i_{j}}) +W( t_{k}) -W( t_{i_{j}}),
\end{equation*}%
and so $\sum_{i=i_{j}+1}^{k}b_{i}\vert A_{i}\vert <W( t_{k})
-W( t_{i_{j}}).$

We need to show now that $\psi _{j}$ attains its minimum
over $C_{j}$ at the point $\lambda _{j}a$,  $a=(
a_{i_{j}+1},...,a_{i_{j+1}}) $.

Consider first the simplex
$ S_j=\mathbb{R}^{i_{j+1} - i_j}_+\cap H_j$,
where $H_j$ is the hyperplane in $\mathbb{R}^{i_{j+1} - i_j}$ given by the
equation $\sum_{i=i_j+1}^{i_{j+1}}|A_i|x_i=W(
t_{i_{j+1}})-W(t_{i_j})$.
Then Lemma \ref{th:2} tells us that $\lambda_j a$ is the unique minimizer of $\psi_j$
over $S_j$.
It remains to show that $\lambda _{j}a\in C_j\subset S_j $. Suppose for the contradiction that $\lambda _{j}a\not\in C_j$. On the other hand, by the previous reasoning, there exists $\bar b\in C_{j}$ that minimizes
$\psi_{j}$ over $C_{j}$. Define $b(\lambda) = \lambda \bar b+( 1-\lambda ) \lambda_{j}a$ for $0\le\lambda\le 1$. Then by Lemma \ref{th:2} and since $\lambda _{j}a\not=\bar b$ we get $\psi_j(\bar b) > \psi_j(\lambda_j a)$. Moreover the strict convexity of $t\mapsto \varphi(d/t)t$ for each $d>0$, implies strict convexity of $\psi_j$.  Therefore for each $0<\lambda < 1$,
\begin{align*}
\psi _{j}( b( \lambda ) )
& < \lambda \psi_j(\bar b) + (1 - \lambda) \psi_j(\lambda_j a)
< \psi_j(\bar b).
\end{align*}
Notice that for every $0<\lambda <1$, $b_{i_j+1}(\lambda) > \dots >
b_{i_{j + 1}}(\lambda)$, and
\begin{equation*}
\sum_{i=i_{j}+1}^{i_{j+1}}b_{i} ( \lambda ) \vert
A_{i}\vert = \lambda\sum_{i = i_j + 1}^{i_{j+1}}\bar b_i |A_i| + (1 -
\lambda) \lambda_j (F(t_{i_{j+1}}) - F(t_{i_j}))= W( t_{i_{j+1}})
-W( t_{i_{j}}).
\end{equation*}
Moreover for each $k = i_j + 1,\dots, i_{j+1}-1$,
\begin{align*}
\sum_{i=i_{j}+1}^k b_i ( \lambda ) \vert A_{i}\vert& =
\lambda\sum_{i = i_j + 1}^k \bar b_i |A_i| + (1 - \lambda) \lambda_j (F(t_k) -
F(t_{i_j})) \\
&<\lambda (W(t_k) - W( t_{i_{j}})) + (1 - \lambda)
\lambda_j (F(t_k) - F(t_{i_j})).
\end{align*}
This implies that for $0<\lambda < 1$ sufficiently close to $1$, $b(\lambda) \in C_j$. Since $\psi_j(b(\lambda)) <
\psi_j(\bar b)$, the element $\bar b$ cannot minimize $\psi_j$ over $C_j$, which gives the desired contradiction.
We have shown therefore that on $\left[ t_{i_{j}},t_{i_{j+1}}\right) $ it
must be $g=\lambda _{j}f$ and since $j$ was arbitrary, the proof is finished.
 \end{proof}

The following algorithm will be crucial for proving the main Theorem \ref%
{th:3} which provides a procedure to obtain a minimizing function $g$ for the
modular $P_{\varphi,w}(f)$.\newline

\textbf{Algorithm A. }\label{alg:A} Let $f=\Sigma _{i=1}^{n}a_{i}\chi
_{A_{i}}$ for some $a_{1}>\dots >a_{n}>0$ and $A_{i}=\left[
t_{i-1},t_{i}\right) $ where $0=t_{0}<t_{1}<\dots <t_{n}<\infty $. Define
first
\begin{equation*}
g_{-1}=f,\ \ \gamma _{0}=\lambda _{0}=\min_{1\leq i\leq n}\Big\{ \frac{%
W\left( t_{i}\right) }{F\left( t_{i}\right) }\Big\} ,\ \ g_{0}=\gamma
_{0}f=\lambda _{0}f,\ \ i_{0}=0.
\end{equation*}%
Then for $j>0$ let
\begin{align}
i_{j}& =\max \Big\{ i>i_{j-1}:\gamma _{j-1}=\frac{W( t_{i})
-W( t_{i_{j-1}}) }{G_{j-2}( t_{i}) -G_{j-2}(
t_{i_{j-1}}) }\Big\} ,  \label{alg 1} \\
\gamma _{j}& =\min_{i_{j}<i\leq n}\Big\{ \frac{W( t_{i})
-W( t_{i_{j}}) }{G_{j-1}( t_{i}) -G_{j-1}(
t_{i_{j}}) }\Big\} ,  \notag \\
g_{j}& =g_{j-1}\chi _{[ 0,t_{i_{j}}) }+\gamma _{j}g_{j-1}\chi _{%
[ t_{i_{j}},t_{n}) }  .  \notag
\end{align}%
Continue the recurrent step until $i_{m}=n$ for some $m$ and denote $%
g^{f}=g_{m-1}$. Clearly $\gamma_j > 1$ for $j= 1,\dots, m-1$, and
\begin{equation*}
g^{f}=\sum_{j=0}^{m-1}\lambda _{j}f\chi _{[ t_{i_{j}},t_{i_{j+1}}%
) },\ \ \ \  \lambda _{j}=\prod_{i=0}^{j}\gamma _{i} . % \label{alg 2}
\end{equation*}%
Hence $\lambda_0 < \lambda_1 < \dots < \lambda_{m-1}$. We also have for $%
j=0,1,\dots ,m-1$,
\begin{align*}
\gamma _{j}& =\frac{W( t_{i_{j+1}}) -W( t_{i_{j}}) }{%
G_{j-1}( t_{i_{j+1}}) -G_{j-1}( t_{i_{j}}) }=\frac{%
W( t_{i_{j+1}}) -W( t_{i_{j}}) }{\gamma
_{j-1}(G_{j-2}( t_{i_{j+1}}) -G_{j-2}( t_{i_{j}}) )}
=\frac{W( t_{i_{j+1}}) -W( t_{i_{j}}) }{%
\prod_{i=0}^{j-1}\gamma _{i}(F( t_{i_{j+1}}) -F(
t_{i_{j}}) )} .
\end{align*}%
Hence
\begin{equation*}
\lambda _{j}=\prod_{i=0}^{j}\gamma _{i}=\frac{W( t_{i_{j+1}})
-W( t_{i_{j}}) }{F( t_{i_{j+1}}) -F(
t_{i_{j}}) } .
\end{equation*}%
It follows that for each $j=0,1,...,m-1$,
\begin{equation*}
G^{f}( t_{i_{j}}) =W( t_{i_{j}}) .
\end{equation*}%
Now we will show that $g^{f}\prec w$. Evidently $g_{0}=\gamma _{0}f\prec w$.
Similarly as in Lemma \ref{lem:5} we can show that $g_{j}=g_{j}^{\ast } $
for each $j$.

Explaining as in Lemma \ref{lem:5} we can show that $g_{j}=g_{j}^{\ast }$.
In fact, since $f$ is decreasing, it is
sufficient to show that $\lambda_{j-1}a_{i_{j}} \geq \lambda _{j}a_{i_{j}+1}$ for each $j=1,...,m-1$. Fix $j=1,...,m-1$. We have $
G_{j-1}( t_{i_{j}}) =G_{j}( t_{i_{j}}) =W(
t_{i_{j}}) $ and $G_{j}( t_{i_{j}+1}) \leq W(
t_{i_{j}+1}) $ by definition of $\gamma _{j}$. Then
\begin{equation*}
w( t_{i_{j}}) ( t_{i_{j}+1}-t_{i_{j}}) \geq W(
t_{i_{j}+1}) -W( t_{i_{j}}) \geq G_{j}(
t_{i_{j}+1}) -G_{j}( t_{i_{j}}) =\lambda
_{j}a_{i_{j}+1}( t_{i_{j}+1}-t_{i_{j}}) .
\end{equation*}%
On the other hand $G_{j-1}( t_{i_{j}-1}) =G_{j}(
t_{i_{j}-1}) \leq W( t_{i_{j}-1}) $, so
\begin{equation*}
\lambda_{j-1}a_{i_{j}}( t_{i_{j}}-t_{i_{j}-1}) =G_{j}( t_{i_{j}})
-G_{j}( t_{i_{j}-1}) \geq W( t_{i_{j}}) -W(
t_{i_{j}-1}) \geq w( t_{i_{j}}) (
t_{i_{j}}-t_{i_{j}-1}) .
\end{equation*}%
Therefore
$
\lambda_{j-1}a_{i_{j}}\geq w( t_{i_{j}}) \geq \lambda _{j}a_{i_{j}+1}.
$
 It remains to prove that $g_{j-1}\prec w$ implies $g_{j}\prec
w $. By (\ref{alg 1}),
\begin{equation*}
g_{j}=g_{j-1}\chi _{[ 0,t_{i_{j}}) }+\gamma _{j}g_{j-1}\chi _{%
[ t_{i_{j}},t_{n}) } .
\end{equation*}%
Therefore for $k\leq i_{j}$,
\begin{equation*}
G_{j}( t_{k}) =G_{j-1}( t_{k}) \leq W(
t_{k}) .
\end{equation*}%
If $k>i_{j}$, then by definition of $\gamma _{j}$,
\begin{equation*}
G_{j}( t_{k}) =G_{j-1}( t_{i_{j}}) +\gamma _{j}(
G_{j-1}( t_{k}) -G_{j-1}( t_{i_{j}}) ) \leq
W( t_{i_{j}}) +W( t_{k}) -W( t_{i_{j}})
=W( t_{k}),
\end{equation*}%
and then by Lemma \ref{lem:3} we have $g_{j}\prec w$, which proves that $%
g^{f}\prec w$. It is also worth to notice that since $\lambda _{0}<\lambda
_{1}<\dots<\lambda_{m-1}$, the function $f/g^{f}$ is decreasing.

\begin{remark} \rm{The function $g^{f}$ produced by Algorithm A is of the
form (\ref{postac}), but the sequence $\left( t_{i_{j}}\right) $ obtained in
this way need not to be maximal in the sense that there may exist $t_{i}\not\in \left(
t_{i_{j}}\right) $ such that $G^{f}\left( t_{i}\right) = \int_0^{t_i} g^f
=W\left( t_{i}\right) $.}
\end{remark}

Now we are ready for our main result describing how to calculate the infimum of $%
P_{\varphi,w}(f)$ for a decreasing simple function $f$.

\begin{theorem}
\label{th:3} Let $\varphi $ be an $N$-function and let $f=\Sigma
_{i=1}^{n}a_{i}\chi _{A_{i}}$ for some $a_{i}$ with $a_{1}>\dots>a_{n}>0$
and $A_{i}=\left[ t_{i-1},t_{i}\right) $ where $0=t_{0}<t_{1}<\dots<t_{n}<%
\infty $. Then the function\textbf{\ } $g^{f}$ produced by Algorithm A is a
minimizing function for $P_{\varphi ,w}\left( f\right) $, that is
\[
P_{\varphi,w}(f) = \Big\|\varphi\Big(\frac{f}{g^f}\Big) g^f\Big\|_1.
\]
The function $g^f$ is independent of $\varphi$ and depends only on $f$ and $w$.
\end{theorem}

\begin{proof}
We divide the proof into two parts.

\textbf{(I)} Assume first that $\varphi$ is
strictly convex. Let $g^{f}$ be produced by Algorithm A. Suppose that a
function $h$ is minimizing as in Theorem \ref{th:1}. We will prove that $%
h=g^{f}.$ This will be done by induction on the number $s$ of steps in
Algorithm A.

\textbf{(a)} Assume first that $s=1$, that is $\min_{1\le i\le
n}\left\{{W(t_i)}/{F(t_i)}\right\}= {{W(t_n)}/{F(t_n)}}$. Then $%
g^f=\lambda_0 f$, with $\lambda_0={{W(t_n)}/{F(t_n)}}$. On the other
hand by Theorem \ref{th:1}, $h=\sum_{j=0}^{p-1} \mu_j f\chi_{[t_{k_j},
t_{k_{j+1}})}$, where $0=k_0<k_1<\dots<k_p=n$ and $H(t_{k_j})=W(t_{k_j})$, $%
j=1,\dots, p$. But since $\lambda_0f\prec w$, Lemma \ref{th:2} shows that $%
\mu_0=\dots=\mu_{p-1}=\lambda_0$, that is $h=\lambda_0 f=g^f$.

\textbf{(b)} Assume now that $s>1$ and that Algorithm A is valid for $ s-1$
steps. We claim first that
\begin{equation}
W( t_{i_{1}}) =H( t_{i_{1}}),  \label{claim}
\end{equation}
where $i_{1}=\max \left\{ i>0:\lambda _{0}=%
{W( t_{i}) }/{F( t_{i}) }\right\} $, $\lambda
_{0}=\min_{1\le i\leq n}\left\{ {W( t_{i}) }/{F(
t_{i}) }\right\} $. Clearly $i_1<n$. If the claim is false then
$H(t_{i_1}) < W(t_{i_1})$. Now since $H(t_i) \le W(t_i)$ for all $i =
1,\dots, n$, two cases are possible:\smallskip

(i) $H\left( t_{i}\right) <W\left( t_{i}\right) $ for each $0<i\leq i_{1}$
or,

(ii) $ H\left( t_{k}\right)  = W\left( t_{k}\right)$ for some $k<i_{1}$ and $%
H\left( t_{i_{1}}\right) <W\left( t_{i_{1}}\right) $.\smallskip

\noindent Case (i): Suppose that $H\left( t_{i}\right) <W\left( t_{i}\right) $ for
each $i\leq i_{1}$. Then by (\ref{postac}), $h\chi _{\left[ 0,t_{m}\right)
}=\lambda f\chi _{\left[ 0,t_{m}\right) }$ with $H\left( t_{m}\right)
=W\left( t_{m}\right) $ for some $\lambda>0$ and $t_{m}>t_{i_{1}}$. Hence $%
\lambda F\left( t_{i_{1}}\right) =H\left( t_{i_{1}}\right) <W\left( t_{i_1}\right) =%
\lambda _{0}F\left( t_{i_{1}}\right) $ and thus $\lambda <\lambda _{0}$.
It follows that
\begin{equation*}
H( t_{m}) =\lambda F( t_{m}) <\lambda _{0}F(
t_{m}) \leq W( t_{m}),
\end{equation*}%
which is a contradiction with $H(t_m) = W(t_m)$.

Case (ii). Suppose that $H\left( t_{i_{1}}\right) <W\left( t_{i_{1}}\right) $
\ and $W\left( t_{k}\right) =H\left( t_{k}\right) $ for some $k<i_{1}$.
Assume that $k$ is the biggest index satisfying those conditions. Since $h$
is assumed to be a minimizing function, by Theorem \ref{th:1} there exist $%
t_{i_{1}}<t_{m}\leq t_{n}$ and $\lambda >0$ such that
\begin{equation}
h\chi _{\left[ t_{k},t_{m}\right) }=\lambda f\chi _{\left[
t_{k},t_{m}\right) }\ \ \text{and}\ \ \ H(t_{m})=W(t_{m}) .  \label{eqB}
\end{equation}%
Since $\lambda_0 F\le W$ and $\lambda_0 F(t_{i_1})=W(t_{i_1})$. We have
\[\lambda(F(t_{i_1})-F(t_k))=H(t_{i_1})-W(t_k)<W(t_{i_1})-W(t_k)\le \lambda_0(F(t_{i_1})-F(t_k))\]
and
\[\lambda(F(t_m)-F(t_{i_1}))=W(t_m)-H(t_{i_1})>\lambda_0 F(t_m)-W(t_{i_1})=\lambda_0(F(t_m)-F(t_{i_1}))\]
which gives the contradiction.
Therefore we have shown the claim (\ref{claim}).

Next we will show that
\begin{equation}  \label{eq:111}
h\chi_{[0,t_{i_1})} = \lambda_0 f\chi_{[0,t_{i_1})} = g^f \chi_{[0,t_{i_1})} .
\end{equation}
Suppose for the contrary that
\begin{equation*}
h\chi _{[ 0,t_{i_{1}}) }=\sum_{j=0}^{r-1}\delta _{j}f\chi _{[
t_{k\left( j\right) },t_{k\left( j+1\right) })} ,
\end{equation*}%
where $0=t_{k\left( 0\right) }<t_{k\left( 1\right) }<...<t_{k\left( r\right)
}=t_{i_{1}}$ with $\delta _{j}\not=\lambda _{0}$ for some $j=0,1,\dots, r-1$%
. Then by Lemma \ref{th:2} applied to the interval $\left[
0,t_{i_{1}}\right) $ and $\lambda_0$, we have
\begin{eqnarray*}
\Vert \varphi ( {f}/{h}) h\Vert _1
&=&\Vert \varphi ( {f}/{h}) h\chi _{[
0,t_{i_{1}}) }\Vert _1+\Vert \varphi ( {f}/{h}
) h\chi _{[ t_{i_{1}},t_{n}) }\Vert _1> \\
&>&\Vert \varphi ( {f}/{(\lambda _{0}f)}) \lambda
_{0}f\chi _{[ 0,t_{i_{1}}) }\Vert _1+\Vert
\varphi ( {f}/{h}) h\chi _{[ t_{i_{1}},t_{n})
}\Vert _1.
\end{eqnarray*}
It follows that $h$ is not minimizing $P_{\varphi,w}(f)$, which contradicts
our assumption and proves (\ref{eq:111}).

Now in view of $H(t_{i_1})=W(t_{i_1})$ we have by the proof of Lemma \ref%
{Decomposition} that $h_{i_1}\left( t\right)=h\chi _{\left[ {i_1},t_n\right)
}\left(t+t_{i_1}\right)$ is a minimizing function in the modular $P_{\varphi
,w_{i_1}}\left( f_{i_1}\right)$ where for $t\in I$,
\begin{equation*}
f_{i_1}\left( t\right)=(f\chi _{[ t_{i_1},t_n)
})(t+t_{i_1}),\ \ \ w_{i_1}( t)=(w\chi _{[
t_{i_1},t_n) })(t+t_{i_1}).
\end{equation*}
On the other hand it is straightforward to see that Algorithm A for $f_{i_1}$, and the weight $w_{i_1}$, has $s-1$ steps and that it yields a function $g^{f_{i_1}}$ which is nothing but the function $g^f\chi_{[t_{i_1}, t_n)}$
shifted backward by $t_{i_1}$, that is
\begin{equation*}
g^{f_{i_1}}(t)=g^f(t+t_{i_1}) .
\end{equation*}
Now by induction hypothesis we have $g^{f_{i_1}}=h_{i_1}$ and thus
\begin{equation*}
g^f\chi_{[t_{i_1}, t_n)}=h\chi_{[t_{i_1}, t_n)} .
\end{equation*}
We know also by Lemma \ref{Decomposition} that $h\chi_{[0,t_{i_1})}$ is
minimizing $P_{\varphi,w}(f\chi_{[0,t_{i_1})})$ while clearly Algorithm A
for $f\chi_{[0,t_{i_1})}$ has only one step and yields $g^f%
\chi_{[0,t_{i_1})} $. Thus by part (a), $g^f\chi_{[0,t_{i_1})}=h%
\chi_{[0,t_{i_1})}.$ Therefore $g^f = h$ and this finishes the proof of case
(I).

\textbf{(II) }Assume now that $\varphi $ is any $N$-function. Let $
\varphi _{m}(t)=\varphi(t)+\frac 1m t^2 $. Then %
the functions  $\varphi_m$ are strictly convex $N$-functions and $%
\varphi_m\to \varphi $ uniformly on compact sets. Let $g^{f}$ be produced by
Algorithm A. Suppose $g^{f}$ is not minimizing for $P_{\varphi ,w}\left(
f\right) $, i.e. there is $h=\Sigma _{i=1}^{n}b_{i}\chi _{A_{i}}\prec w$,
such that for some $\delta >0$ we have
\begin{equation*}
\left\Vert \varphi \left( {f}/{h}\right) h\right\Vert _1+\delta
\leq \Vert \varphi ( {f}/{g^{f}}) g^{f}\Vert
_1.
\end{equation*}%
Since $h$, $f$ and $g^f$ are simple functions,

\begin{equation*}
\left\Vert \varphi _{m}\left( {f}/{h}\right) h\right\Vert
_1\rightarrow \left\Vert \varphi \left( {f}/{h}\right)
h\right\Vert _1\ \ \text{and} \ \ \ \Vert \varphi _{m}(
{f}/{g^f}) g^f\Vert _1\rightarrow \Vert \varphi
( {f}/{g^f}) g^f\Vert _1.
\end{equation*}%
Hence there is $N$ such that for $m>N$,
\begin{equation*}
\left\Vert \varphi _{m}\left( {f}/{h}\right) h\right\Vert _1\leq
\left\Vert \varphi \left( {f}/{h}\right) h\right\Vert _1+
\delta/{3}\ \ \text{and} \ \ \Vert \varphi _{m}( {f}/{g^{f}}%
) g^{f}\Vert _1\geq \Vert \varphi ( {f}
/g^{f}) g^{f}\Vert _1-{\delta }/{3} .
\end{equation*}%
Therefore for each $m>N$,
\begin{equation*}
\left\Vert \varphi _{m}\left( {f}/{h}\right) h\right\Vert _1+
\delta/{3}\leq \Vert \varphi _{m}( {f}/{g^{f}})
g^{f}\Vert _1,
\end{equation*}
which means that $g^{f}$ does not minimize $P_{\varphi _{m},w}\left(
f\right) $ and so it contradicts the case (I) and the proof is completed.
\end{proof}

\bigskip

\section{\protect\bigskip Dual norms of $\Lambda_{\protect\varphi,w}$ in
terms of level functions.}

In this section we develop formulas for K\"othe  duals of Orlicz-Lorentz spaces equipped with Luxemburg or Amemiya norms in terms of level functions.    Let $w$ be a weight function on $I$ as defined in introduction. For $%
f=f^{\ast }$ locally integrable on $I$, define after Halperin \cite{Ha53}
for $0\leq a<b<\infty $, $a,b\in I$,
\begin{equation*}
W( a,b) =\int_{a}^{b}w ,\ \ \ F( a,b) =\int_{a}^{b}f ,\
\ \ R( a,b) =\frac{F( a,b) }{W( a,b) } ,\ \
\end{equation*}%
and for $b=\infty $,
\begin{equation*}
R(a,b)=R(a,\infty )=\limsup_{t\rightarrow \infty }R( a,t) .
\end{equation*}%
Then $( a,b) \subset I$ is called a \textit{
level interval (resp. degenerate level interval) of $f$ with respect to $w$} if $b<\infty $ (resp. $b=\infty $%
) and for each $t\in ( a,b) $,
\begin{equation*}
R( a,t) \leq R( a,b) \text{ and }0<R( a,b) .
%\label{def of l.i.}
\end{equation*}%
It is easy to see that the restriction $0<R\left( a,b\right) $ ensures that any level
interval of $f=f^{\ast }$ is in fact included in the support of $f^{\ast }$, and
this is the only difference with the original definition from \cite{Ha53}.
Level interval can be equivalently considered as open, closed or
half-closed. If the weight $w$ is fixed then we will say \textit{level
interval of $f$}, or just l.i. for simplicity. If a level interval is not
contained in any larger level interval, then it is called \textit{maximal
level interval of $f$ with respect to $w$}, or just maximal level interval
and in short m.l.i. In \cite{Ha53}, Halperin proved that maximal level
intervals of $f$ with respect to $w$ are pairwise disjoint and unique and
therefore there is at most countable number of maximal level intervals.

\begin{remark}
\label{rem:4.1}

%\textbf{\ Zmieni\l em def. l.i. wi\k{e}c tego przypadku nie
%ma. Ew. mo\.{z}na to inaczej skomentowa\'{c}. }\textrm{(1) If $\supp f=(0,s)$%
%, then the interval $(s,\infty )$ is a degenerate l.i., and it is also
%m.l.i. whenever $w$ is positive (as in our case). }

\rm{(1) The whole semiaxis $(0,\infty )$ may be a degenerate l.i.. Take
for example any weight function $w$ and let $f=w$. }

\rm{(2) Given any weight $w$, if a decreasing function $f$ is constant on $%
(a,b)$ then $(a,b)$ is a l.i. of $f$ with respect to $w$. }

\rm{First we make a simple observation that the function $t\mapsto
(\int_a^t h)/(t-a)$ is decreasing for $t>a$ whenever $h$ is decreasing and
locally integrable. Letting now $f(t) = c$ for $t \in (a,b)$, by the fact that $w$ is decreasing,  the inequality
$R(a,t) \le R(a,b)$ on $(a,b)$ is equivalent to $\frac{1}{b-a}\int_a^bw \le
\frac{1}{t-a}\int_a^t w$ on $(a,b)$.  }

\rm{(3) If $w$ is a constant weight  then $(a,b)$ is an l.i. of $f$ if and only if
$f$ is constant on $(a,b)$. Consequently any decreasing function with
countable many different values has infinite many m.l.i. with respect of
a constant weight. }

\rm{Let now $w$ be constant on $I$, and $(a,b)$ be a l.i. of $f$ with respect to $%
w $. Therefore $F(a,t)/(t-a)\leq F(a,b)/(b-a)$ on $(a,b)$, and since $f$ is
decreasing we have the equality, that is $F(a,t)= F(a,b)(t-a)/(b-a)$ on
$(a,b)$. Hence $f(t)=F(a,b)/(b-a)$ for all $t\in (a,b)$, and so $f$ is
constant on $(a,b)$.}

\end{remark}

\begin{definition}
\cite{Ha53} Let $f\in L^{0}$ be decreasing and locally integrable on $I$.
Then the \textit{level function} $f^{0}$ of $f$ with respect to $w$ is
defined as
\begin{equation*}
f^{0}\left( t\right) =\left\{
\begin{array}{cc}
R\left( a_{n},b_{n}\right) w\left( t\right) & \text{for }t\in (
a_{n},b_{n}) , \\
f\left( t\right) & \text{otherwise} ,
\end{array}%
\right.
\end{equation*}%
where $\left( a_{n},b_{n}\right) $ is an enumeration of all maximal level
intervals of $f$.
\end{definition}

\begin{lemma}
\label{Lemat ciag prostych}Let $f=f^{\ast }=\Sigma _{i=1}^{n}a_{i}\chi
_{A_{i}}\in \mathcal{M}_{\varphi ,w}$, where $A_{i}=[t_{i-1},t_{i})$ and $%
0=t_{0}<t_{1}<\dots <t_{n}<\infty $. Then
\begin{equation*}
P_{\varphi ,w}\left( f\right) =\int_{I}\varphi \Big( \frac{f}{g^{f}}\Big)
g^{f}=\int_{I}\varphi \Big( \frac{f^{0}}{w}\Big) w .
\end{equation*}%
In particular, the intervals $(t_{i-1},t_{i})$ are level intervals of $f$ with respect to $w$.
Moreover, the maximal level intervals of $f$ with
respect to $w$ are the $(t_{i_{j}},t_{i_{j+1}})$, where
\begin{equation*}
g^{f}=\sum_{j=0}^{m-1}\lambda _{j}f\chi _{\lbrack t_{i_{j}},t_{i_{j+1}})}
\label{mli Hal}
\end{equation*}%
is from Algorithm A.
\end{lemma}

\begin{proof}
Let $f=\Sigma _{i=1}^{n}a_{i}\chi _{A_{i}}$ with $A_{i}=[t_{i-1},t_{i})$, $%
0=t_{0}<t_{1}<\dots <t_{n}<\infty $, and
\begin{equation*}
g^{f}=\sum_{j=0}^{m-1}\lambda _{j}f\chi _{\lbrack t_{i_{j}},t_{i_{j+1}})}
\label{comp 1H}
\end{equation*}%
be as in Algorithm A, where $\lambda_0 < \lambda_1 < \dots <\lambda_{m-1}$
and
\begin{equation*}
\lambda _{j}=\frac{W( t_{i_{j}},t_{i_{j+1}}) }{F(
t_{i_{j}},t_{i_{j+1}}) } ,\ \ \ j=0,1\dots,m-1 .
\end{equation*}%
Hence by Theorem \ref{th:3},
\begin{equation}
P_{\varphi ,w}\left( f\right)
=\sum_{j=0}^{m-1}\int_{t_{i_{j}}}^{t_{i_{j+1}}}\varphi \Big( \frac{f}{%
\lambda _{j}f}\Big) \lambda _{j}f=\sum_{j=0}^{m-1}\varphi \Big( \frac{%
F( t_{i_{j}},t_{i_{j+1}}) }{W( t_{i_{j}},t_{i_{j+1}}) }%
\Big) W( t_{i_{j}},t_{i_{j+1}}) .  \label{mod H1}
\end{equation}
We will now compute the level function $f^{0}$ with respect to $w$. Suppose
first that
\begin{equation*}
w=Tw=\sum_{i=1}^{n}\Big( \frac{1}{|A_{i}|}\int_{A_{i}}w\Big) \chi
_{A_{i}} .
\end{equation*}%
We shall show that every $(t_{i_{j}},t_{i_{j+1}})$ is a maximal level
interval of $f$ with respect to $w$. By Remark \ref{rem:4.1} each $%
(t_{i},t_{i+1})$ is a level interval of $f$. Moreover one can check that on
each $\left( t_{i},t_{k}\right) $, $i<k\leq n$,
\begin{equation*}
\forall _{t\in \left( t_{i},t_{k}\right) }\ \ \frac{F\left( t_{i},t\right) }{%
W\left( t_{i},t\right) }\leq \frac{F\left( t_{i},t_{k}\right) }{W\left(
t_{i},t_{k}\right) }\Leftrightarrow \forall _{i<j<k}\ \ \frac{F\left(
t_{i},t_{j}\right) }{W\left( t_{i},t_{j}\right) }\leq \frac{F\left(
t_{i},t_{k}\right) }{W\left( t_{i},t_{k}\right) } .  \label{Hal 1}
\end{equation*}%
Let us show that each interval $\left( t_{i_{j}},t_{i_{j+1}}\right) $ is a
level interval for $f$ with respect to $w$. In fact we need only to show
that
\begin{equation*}
R\left( t_{i_{j}},t_{k}\right) \leq R\left( t_{i_{j}},t_{i_{j+1}}\right)
\text{ for each }i_{j}<k<i_{j+1} .
\end{equation*}%
Applying the notation of Algorithm A, see (\ref{alg 1}), we have for $j=0,1,\dots, m-1$,
\begin{align*}
\lambda _{j}&=\prod_{i=0}^{j}\gamma _{i}=\frac{W( t_{i_{j+1}})
-W( t_{i_{j}}) }{F( t_{i_{j+1}}) -F(
t_{i_{j}}) }=\frac{1}{R( t_{i_{j}},t_{i_{j+1}}) } , \\
g_{j-1} & = \lambda_0 f\chi_{[0,t_{i_1})} + \lambda_1 f\chi_{[t_{i_1},
t_{i_2})} + \dots + \lambda_{j-2} f\chi_{[t_{i_{j-2}}, t_{i_{j-1}})} +
\lambda_{j-1} f \chi_{[t_{i_{j-1}}, t_n)} .
\end{align*}
Hence for $i_j < k < i_{j+1}$, $G_{j-1} (t_k) - G_{j-1}(t_{i_j}) =
\lambda_{j-1} (F(t_k) - F(t_{i_j}))$, and so
\begin{eqnarray*}
\frac{1}{R( t_{i_{j}},t_{i_{j+1}}) } &=&\lambda _{j}=\gamma
_{j}\lambda _{j-1}=\lambda _{j-1}\min_{i_{j}<i\leq n}\Big\{ \frac{W(
t_{i}) -W( t_{i_{j}}) }{G_{j-1}( t_{i})
-G_{j-1}( t_{i_{j}}) }\Big\} \leq \\
&\leq &\lambda _{j-1}\frac{W( t_{k}) -W( t_{i_{j}}) }{%
G_{j-1}( t_{k}) -G_{j-1}( t_{i_{j}}) }=\frac{W(
t_{k}) -W( t_{i_{j}}) }{F( t_{k}) -F(
t_{i_{j}}) } =\frac{1}{R( t_{i_{j}},t_{k}) } ,
\end{eqnarray*}%
which proves that $( t_{i_{j}},t_{i_{j+1}}) $ is a level interval.

To see that each $( t_{i_{j}},t_{i_{j+1}}) $ is a maximal level
interval we will need Theorem 3.1 from \cite{Ha53}, which states that if $%
a_{1}<a_{2}<b_{1}<b_{2}$ and $( a_{1},b_{1}) ,(
a_{2},b_{2}) $ are level intervals of $f$ with respect to $w$, then $%
( a_{1},b_{2}) $ is also a level interval of $f$ with respect to $%
w$. We also need the simple observation that $\left( a,b\right) $ is a level
interval if and only if
\begin{equation} \label{relsimple}
R( a,b) \leq R( s,b) \text{ for each }s\in (
a,b)  .
\end{equation}%
The latter is a result of the elementary inequalities that for $v,x,y,z>0$,
\begin{equation*}
\frac{y}{z}\leq \frac{v+y}{x+z}\Longleftrightarrow \frac{v+y}{x+z}\leq \frac{%
v}{x},
\end{equation*}%
and that
\begin{equation*}
R\left( a,b\right) =\frac{F\left( a,s\right) +F\left( s,b\right) }{W\left(
a,s\right) +W\left( s,b\right) } .
\end{equation*}%
Suppose therefore that $( t_{i_{j}},t_{i_{j+1}}) $ is not
maximal. Then there is another level interval $\left( a,b\right) $ such that
$( t_{i_{j}},t_{i_{j+1}}) \varsubsetneq \left( a,b\right) $. It
follows that $a<t_{i_{j}}$ or $t_{i_{j+1}}<b$. Suppose $t_{i_{j+1}}<b$ (in
the other case the proof is similar). Then by the mentioned Halperin's
result we get that\ $( t_{i_{j}},t_{i_{j+2}}) $ is a level
interval. But then by definition of level intervals we get
\begin{equation*}
\frac{1}{\lambda _{j}}=R( t_{i_{j}},t_{i_{j+1}}) \leq R(
t_{i_{j}},t_{i_{j+2}}) {\leq }R( t_{i_{j+1}},t_{i_{j+2}}) =%
\frac{1}{\lambda _{j+1}} ,
\end{equation*}%
which means that $\lambda _{j}\geq \lambda _{j+1}.$ However by Algorithm A,
$\lambda _{j+1}=\gamma _{j+1}\lambda _{j}$ with $\gamma _{j+1}>1$, which
gives a contradiction.

Let now $w$ be arbitrary. Denote by $TW(t) = \int_0^t Tw$. Notice that $%
TW(t_i) = W(t_i)$ for each $i$, and $TW(t) \le W(t)$ for any $t>0$. The
latter holds since for any $t\in (t_{k-1}, t_k)$,
\begin{equation*}
TW(t) = W(t_{k-1}) + \Big(\frac{1}{t_k - t_{k-1}} \int_{t_{k-1}}^{t_k} w
\Big) (t-t_{k-1}) \le W(t_{k-1}) + \int_{t_{k-1}}^t w = W(t) .
\end{equation*}
Then for each $j$ and each $t\in \left( t_{i_{j}},t_{i_{j+1}}\right) $ one
has $W(t_{i_{j}},t)=W\left( t\right) -W\left( t_{i_{j}}\right) \geq TW\left(
t\right) -TW\left( t_{i_{j}}\right) =TW(t_{i_{j}},t)$. Therefore, since by the
first part of the proof $(t_{i_j}, t_{i_{j+1}})$ is a l.i. of $f$ with respect to $Tw$%
,
\begin{equation*}
R( t_{i_{j}},t) =\frac{F( t_{i_{j}},t) }{W(t_{i_{j}},t)%
}\leq \frac{F( t_{i_{j}},t) }{TW(t_{i_{j}},t)}\leq \frac{F(
t_{i_{j}},t_{i_{j+1}}) }{TW(t_{i_{j}},t_{i_{j+1}})}=\frac{F(
t_{i_{j}},t_{i_{j+1}}) }{W(t_{i_{j}},t_{i_{j+1}})}=R(
t_{i_{j}},t_{i_{j+1}}) ,
\end{equation*}%
which shows that also $(t_{i_j}, t_{i_{j+1}})$ is l.i. of $f$ with respect
to $w$. By the previous reasoning it is also m.l.i. of $f$ with respect to $%
w $.

Thus the level function $f^{0}$ with respect to $w$ is given by
\begin{equation*}
f^{0}\left( t\right) =
\begin{cases}
&R( t_{i_{j}},t_{i_{j+1}}) w\left( t\right)
\text{ for }t\in ( t_{i_{j}},t_{i_{j+1}}) , \ \ j=0,1,\dots, m-1 ,\\
&0 \ \ \text{for}\ \  t\ge t_n .
\end{cases}
\end{equation*}%
 Then, by (\ref{mod H1}),
\begin{equation*}
\int_{I}\varphi ( {f^{0}}/{w}) w=\sum_{j=0}^{m-1}\varphi
(R( t_{i_{j}},t_{i_{j+1}}) )W( t_{i_{j}},t_{i_{j+1}})
=P_{\varphi ,w}( f) ,
\end{equation*}
and the proof is finished.
\end{proof}

\begin{remark}
\label{rem:4.3} \rm{Algorithm A and Lemma \ref{Lemat ciag prostych} suggest also
another point of view. Namely, rather than changing the function $f$, we may
change the weight according to definition of $P_{\varphi ,w}(f)$. Let's
define \textit{inverse level function} of $w$ with respect to a decreasing
function \thinspace $f$ as follows
\begin{equation*}
w^{f}\left( t\right) =\left\{
\begin{array}{cc}
f\left( t\right)/R\left( a_n,b_n\right) & \text{for }t\in (a_n,b_n) , \\
w\left( t\right) & \text{otherwise} ,%
\end{array}%
\right.  %\label{inv l f}
\end{equation*}
where $(a_n,b_n)$ is an enumeration of all maximal level intervals of $f$ with
respect to $w$. Then by definition of $w^{f}$ the following equality holds
\begin{equation*}
\int_{I}\varphi \Big( \frac{f^{0}}{w}\Big) w=\int_{I}\varphi \Big( \frac{%
f}{w^{f}}\Big) w^{f} .  %\label{second formula}
\end{equation*}%
Notice also that we have $w^{f}\prec w$. In fact, for each m.l.i. $\left(
a,b\right) $ of $f$ with respect to $w$ one has $W(a,b)=W^{f}(a,b) =
\int_a^b w^f$. Moreover, for $t\in \left(a,b\right) $,
\begin{equation*}
W^{f}(a,t)=\frac{F(a,t)}{R\left( a,b\right)} \leq \frac{F(a,t)}{R\left(
a,t\right)}=W(a,t) .
\end{equation*}%
  If $t$ is out of any m.l.i. then the equality $W(t)=W^{f}(t) $ holds. Indeed
 \[
 W^{f}(t)=\sum_{b_n\le t}W^{f}(a_n,b_n)+\int_{E_t}w=  \sum_{b_n\le t}W(a_n,b_n)+\int_{E_t} w=  W(t) ,
 \]
where $E_t=(0,t)\setminus\bigcup_{b_n\le t} (a_n,b_n)$. When $t$ is in some  m.l.i. then we have only $W^{f}(t)\le W(t)$.}
\end{remark}

The next result is a representation of the modular $P_{\varphi ,w}\left( f\right)$ via level function of $f^*$ in case when its support is a finite interval.

\begin{proposition}
\label{prop: level represent} Let $\varphi $ be an $N$-function. Then for any $%
f=f^*\in \mathcal{M}_{\varphi ,w}$ such that $\supp f =(0,s) $ where $s<\infty $,
we have
\begin{equation*}
P_{\varphi ,w}\left( f\right) =\int_{I}\varphi \Big( \frac{f^{0}}{w}\Big)
w=\int_{I}\varphi \Big( \frac{f}{w^{f}}\Big) w^{f} .
\end{equation*}
\end{proposition}

\begin{proof}
Let\ $f=f^{\ast }\in \mathcal{M}_{\varphi ,w}$
and let $(C_j)$ be an enumeration
 of all m.l.i. of $f$ with respect to $w$.
 For every $n\in \mathbb N$, let $\mathcal D_n=\{(k2^{-n}s,(k+1)2^{-n}s]: 0\le k< 2^n\}$ be the set of dyadic subdivisions of the interval $(0,s]$ and $\mathcal C_n=\{C_j:j\le n\}$.  The endpoints of all the intervals in $\mathcal D_n\cup \mathcal C_n$, when rearranged in increasing order, define a finite partition $\mathcal A_n$ of the interval $(0,s]$ into subintervals $A^n_k, k=1,\dots K(n)$. In other words,  $\mathcal A_n$ and $\mathcal D_n\cup \mathcal C_n$ generate the same algebra $\mathcal F_n$ of subsets of $(0,s]$. Moreover
\begin{equation}\label{ineq:4.6}
|A_k^n| \le \frac{s}{2^n}\ , \ \ \ k= 1,\dots,K(n) .
\end{equation}
Then for each $j\in \mathbb{N}$ and $n\ge j$ there is a finite set $%
I(j,n) \subset \mathbb{N}$ such that
\begin{equation*}
C_j = \bigcup_{k\in I(j,n)} A_k^n .
\end{equation*}

Since $\mathcal{M}_{\varphi,w} \subset L^1 + L^\infty$ \cite{BS88} and so $f$ is integrable on $[0,s)$, we may define for each $n$ the simple function
\begin{equation*}
f_{n}=\sum_{k=1}^{K(n)}\Big( \frac{1}{|A_{k}^{n}|}\int_{A_{k}^{n}}f\Big)
\chi _{A_{k}^{n}} ,
\end{equation*}%
 which is the conditional expectation of $f$ with respect to the algebra $\mathcal F_n$.

We will show that $f_{n}^{0}\rightarrow f^{0}$ a.e., where $f_{n}^{0}$, $f^{0}$
are level functions for $f_{n}$, $f$, respectively. Fix some m.l.i. $%
C_{j}=(d,e] $ of $f$ with respect to $w$. Then $\ (d,e]=\bigcup_{k\in
I(j,n)}A_{k}^{n}$ for each $n\geq j$. Thus since $f$ is decreasing and as in
Remark \ref{rem:4.1}(2), for each $t\in (d,e]$,
\begin{equation*}
F_{n}(d,t)\leq F(d,t)\text{, \ }F_{n}(d,e)=F(d,e) .
\end{equation*}%
Therefore $(d,e]$ is an l.i. of $f_{n}$ with respect to $w$ for all $n\geq j$%
. Clearly for each $n\geq j$, the set $C_{j}$ is contained in  some m.l.i. $C^{n}
= (d_n, e_n]$ of $f_{n}$. We claim that
\begin{equation}
\left\vert C^{n}-C_{j}\right\vert \rightarrow 0\text{ as }n\rightarrow
\infty  .  \label{claim 1}
\end{equation}%
In fact if (\ref{claim 1}) does not hold, there exist a subsequence $\left(
n_{k}\right) $ and numbers $d_{0},e_{0}\in \left[ 0,s\right] $ such that $%
d_{n_{k}} \to d_0$ and $e_{n_{k}}\to e_0$, and $d_{0}<d$ or $e<e_{0}$.
Moreover $(d_n, e_n]$ is a union of some intervals $A_k^n$ and so $R_n(d_n,
e_n) = R(d_n, e_n)$, where $R_{n}$ is just defined as
\begin{equation*}
R_{n}(s,u)={F_{n}(s,u)}/{W(s,u)} .
\end{equation*}
For each $t\in (d_{0},e_{0})$ we have
that $t\in (d_{n_{k}},e_{n_{k}}]$ for large $k$.
Choose $A_{i(n)}^n=(w_n,v_n]$ and $A_{j(n)}^n=(p_n,r_n]$ in such a way that $t\in A_{i(n)}^n$ and $d_0\in A_{j(n)}^n$ for each $n$. If $d_0<d_{n_k}$, then
\begin{align*}
& |F(d_0,t)-F_{n_k}(d_{n_k},t)| \\
&  \leq |F(d_0,d_{n_k})|+|F(d_{n_k},w_{n_k})-F_{n_k}(d_{n_k},w_{n_k})|+|F(w_{n_k},t)-F_{n_k}(w_{n_k},t)|
\\&
\leq \int_{p_{n_k}}^{d_{n_k}}f + \int_{A_{i(n_k)}^{n_k}}f + \int_{A_{i(n_k)}^{n_k}}f_{n_k}=\int_{p_{n_k}}^{d_{n_k}}f + 2\int_{A_{i(n_k)}^{n_k}}f .
\end{align*}
Similarly, if  $d_0>d_{n_k}$,
\[
|F(d_0,t)-F_{n_k}(d_{n_k},t)|
 \leq \int_{d_{n_k}}^{r_{n_k}}f + 2\int_{A_{i(n_k)}^{n_k}}f .
\]
In consequence, from both cases in view of (\ref{ineq:4.6}) we get
\[
|F(d_0,t)-F_{n_k}(d_{n_k},t)|\rightarrow 0 ,
\]
and so
\begin{equation*}
R(d_{0},t)\longleftarrow R_{n_k}(d_{n_{k}},t)\leq
R_{n_k}(d_{n_{k}},e_{n_{k}}) = R(d_{n_{k}},e_{n_{k}})\rightarrow
R(d_{0},e_{0}) .
\end{equation*}%
It follows that $(d_{0},e_{0})$ is a l.i. of $f$, which contradicts our
assumption on maximality of $C_{j}$ and proves (\ref{claim 1}).

Let $t\in C_{j}$ for some $j$. Then keeping notation like above we have
\begin{equation}
f_{n}^{0}(t)=R_{n}(d_{n},e_{n})w(t)=R(d_{n},e_{n})w(t)\rightarrow
R(d,e)w(t)=f^{0}(t) .  \label{ae 1}
\end{equation}%
Suppose now $t\in \lbrack 0,s)\backslash \bigcup_{j}\overline{C_{j}}$. Then
for all $n\in\mathbb{N}$ there exists $k_0 = k_0(n)$ such that $t\in
A_{k_{0}}^{n}$. Since $A_{k_{0}}^{n}$ are l.i. of $f_{n}$, there are m.l.i. $%
M^{n}$ of $f_{n} $ such that $A_{k_{0}}^{n}\subset M^{n}=(m_{n},h_{n}]$.
Clearly $(m_n, h_n]$ is a union of some sets $A_k^n$. One may also explain
like in (\ref{claim 1}) that $|M^n| \to 0$ as $n\to\infty$, and so for a.a. $t$,
\begin{equation}
f_{n}^{0}(t)=R_{n}(m_{n},h_{n})w(t)=R(m_{n},h_{n})w(t)=\frac{%
F(m_{n},h_{n})/\left\vert M^{n}\right\vert }{W(m_{n},h_{n})/\left\vert
M^{n}\right\vert }w(t)\rightarrow f(t) .  \label{ae 2}
\end{equation}%
Thus we get from (\ref{ae 1}) and (\ref{ae 2}) that $%
f_{n}^{0}\rightarrow f^{0}$ a.e..

Notice that $P_{\varphi ,w}\left( f_{n}\right) \leq P_{\varphi
,w}\left( f\right) $. In fact, suppose $P_{\varphi ,w}\left( f\right) =k$.
Consider the space $\mathcal{M}_{\psi ,w}$, where $\psi (t)=\varphi (t)/k$,
with the Luxemburg norm $\left\Vert \cdot \right\Vert $ given by the modular
\begin{equation*}
P_{\psi ,w}\left( f\right) =\frac{1}{k}P_{\varphi ,w}\left( f\right) \text{.
}
\end{equation*}%
This is a r.i. Banach function space with the Fatou property by Proposition %
\ref{prop:1}. Since $f_n\prec f$ we have
$\left\Vert f_{n}\right\Vert_{\mathcal{M}_{\psi ,w}} \leq \left\Vert
f\right\Vert_{\mathcal{M}_{\psi ,w}} =1\ \text{for each }n\text{.}$
It follows from the left continuity of the function $(0,\infty) \backepsilon\lambda\mapsto P_{\psi,w}(\lambda f)$ (see Lemma 4.6 in \cite{KR12})
that $P_{\psi ,w}\left( f_{n}\right)\leq  1$,
 and so
\begin{equation*}
P_{\varphi ,w}\left( f_{n}\right) \leq P_{\varphi ,w}( f) .
\end{equation*}
Applying this, the convergence $f_{n}^{0}\rightarrow f^{0}$ a.e. and $w^f
\prec w$ by Remark \ref{rem:4.3}, we get
\begin{eqnarray*}
P_{\varphi ,w}( f) &\geq &\liminf P_{\varphi ,w}(
f_{n}) \overset{\text{Lemma \ref{Lemat ciag prostych}}}{=}\liminf
\int_{I}\varphi ( {f_{n}^{0}}/{w}) w  \label{ae 3} \\
&\overset{\text{Fatou Lemma}}{\geq }&\int_{I}\varphi ( {f^{0}}/{w}%
) w =\int_{I}\varphi ( {f}/{w^{f}}) w^{f}\geq
P_{\varphi ,w}( f) ,
\end{eqnarray*}
which finishes the proof.
\end{proof}

\begin{lemma}\label{lem:degenerate}
Let $\varphi $ be an $N$-function and $W(\infty )=\infty $. If $%
f=f^*\in \mathcal{M}_{\varphi ,w}$ then it does not have any degenerate level interval.
\end{lemma}

\proof
 Suppose there is a degenerate m.l.i. $(a,\infty )$ of $f$, that is
\begin{equation*}
R(a,t)\leq \limsup_{x\rightarrow \infty }R(a,x)= R(a,\infty )\text{ for each
}t>a\text{,} \ \text{and}\ \  R(a,\infty )>0 .
\end{equation*}%
 Without loss of generality we also suppose that $P_{\varphi,w}(f) < \infty$. We can do this since level
intervals of $f$ are the same for all $kf$, where $k>0$.

We will consider three cases.

a) Suppose $R(a,t)<\limsup_{x\rightarrow \infty }R(a,x)$ for each $t>a$.
Define%
\begin{equation*}
x_{n}=\max \{ x\in [a,a+n]:R(a,x)=\sup \{R(a,t): t\in [a,a+n]\}\} .
\end{equation*}%

We have that $x_n \nearrow \infty$ and $R(a,\infty) = \lim_{n\to\infty} R(a,x_n)$. In fact if $x_n \to x_0 < \infty$ then by the assumption $R(a,x_0) = \lim_{n\to\infty} R(a,x_n) = \sup_{t\in (a,\infty)} R(a,t) < R(a,\infty)$, which is impossible. Therefore $x_n \nearrow \infty$ and $\lim_{n\to\infty} R(a,x_n) = \sup_{t\in (a,\infty)} R(a,t) = \limsup_{t\to\infty} R(a,t) = R(a,\infty)$.

Consider the sequence of functions $%
g_{n}=f\chi _{( 0,x_{n}] }$.
Clearly $R(a,t)\leq R(a,x_{n})$ for each $a<t<x_{n}$. Hence $(a,x_n]$ is an l.i. of $f$ and thus it is an m.l.i. of $g_n$. Therefore $g_n^0 = f^0\chi_{(0,a)} + R(a,x_n) w \chi_{[a,x_n]} \to f^0\chi_{(0,a)} + R(a,\infty)w \chi_{[a,\infty)} = f^0$, and by Proposition \ref{prop: level represent} applied to $g_n$ we have
\begin{eqnarray*}
P_{\varphi ,w}( g_{n}) &=&\int_{0}^{x_{n}}\varphi (
g_{n}^{0}/{w}) w= \\
&=&\int_{0}^{a}\varphi ( {f^{0}}/{w})
w+\int_{a}^{x_{n}}\varphi ( R(a,x_{n})) w\\
&=&\int_{0}^{a}\varphi ({f^{0}}/{w})
w+\varphi ( R(a,x_{n})) (W(x_n)- W(a))
\rightarrow \infty  ,
\end{eqnarray*}%
since $\int_{a}^{\infty }w=\infty $ by the assumption $W(\infty)=\infty$. On the other hand $%
P_{\varphi ,w}\left( g_{n}\right) \leq P_{\varphi ,w}\left( f\right) $ and
so $P_{\varphi ,w}\left( f\right) =\infty $, which is a contradiction to our assumption.

Consider now the following set
 \[
 B=\{z>a:\  R(a,z)=R(a,\infty )\} .
 \]
If the case a) is not satisfied then $B\ne \emptyset$.

b) Let first $\sup B=\infty $. Then there exists $a< x_n \nearrow \infty$ such that $R(a,x_n) = R(a,\infty)$ for each $n\in\mathbb{N}$, and we proceed as in a).

c) Suppose now that $\sup B=b<\infty $. Clearly $R\left( a,b\right)
=R(a,\infty )$. Let $b<y_{n}\nearrow \infty $ be such that $%
R(a,y_{n})\nearrow R(a,\infty )$. Then for each $\sigma >1$ there exists $N$
such that for $n>N$ we have
\begin{equation*}
R( a,y_{n}) \leq R( a,b) \leq \sigma R(
a,y_{n}).
\end{equation*}
We will show that for sufficiently large $n$,
\begin{equation}\label{ineq}
R(b,y_{n})\leq R( a,y_{n}) \leq \sigma R( b,y_{n}).
\end{equation}%
The left side of this inequality follows immediately from (\ref{relsimple}). In order to get the right side notice first that
\begin{equation*}
\frac{F(a,b)}{W(a,b)}=R\left( a,b\right) \leq \sigma R\left( a,y_{n}\right)
=\sigma \frac{F(a,b)+F(b,y_{n})}{W(a,b)+W(b,y_{n})} .
\end{equation*}%
Then
\begin{equation*}
F(a,b)W(b,y_{n})\leq \sigma F(b,y_{n})W(a,b)+(\sigma -1)F(a,b)W(a,b) ,
\end{equation*}%
and since $W(b,y_{n})\rightarrow \infty $,
\begin{eqnarray*}
F(a,b)W(b,y_{n}) \leq \sigma F(b,y_{n})W(a,b)+(\sigma -1)F(b,y_{n})W(b,y_{n})
\end{eqnarray*}%
for $n$ large enough. Hence
\begin{equation*}
F(a,b)W(b,y_{n})+F(b,y_{n})W(b,y_{n})\leq \sigma \lbrack
F(b,y_{n})W(a,b)+F(b,y_{n})W(b,y_{n})]
\end{equation*}%
and so%
\begin{equation*}
R\left( a,y_{n}\right) =\frac{F(a,b)+F(b,y_{n})}{W(a,b)+W(b,y_{n})}\leq
\sigma \frac{F(b,y_{n})}{W(b,y_{n})}= \sigma R(b,y_{n}) ,
\end{equation*}
and the inequality (\ref{ineq}) is proved.

Therefore $R\left( b,y_{n}\right) \rightarrow R\left( a,b\right) =R(a,\infty
) =R\left( b,\infty \right) $. Moreover,
once again using (\ref{relsimple}) for each $b<t$ from $R\left( a,t\right)
<R\left( a,b\right) $ we have
\[
R\left( b,t\right) <R\left( a,t\right)
<R\left( a,b\right) =R(a,\infty ) = R(b,\infty) ,
\]
where the second inequality follows from
definition of $B$. Therefore  $(b,\infty )$
is an l.i. of $f$. Notice also that $(b,\infty )$ is of the same type as the interval $(a,\infty)$  in the case a).
Choosing $\left( x_{n}\right) $ like in that case for $b$ instead of $a$ we define
$g_{n}=f\chi _{\left[ 0,x_{n}\right] }$. Then m.l.i. of $g_{n}$ are the same
like for $f$ in the interval $\left[ 0,a\right] $. Moreover, by the assumption $\sup B = b < \infty$,  the interval $(a,b]$ is an m.l.i. of $g_n$, and by definition of $(x_n)$, the interval $(b,x_n]$ is an m.l.i. of $g_n$. Hence the sequence $(g_n)$ is increasing and
\[
g_n^0 = f^0 \chi_{(0,a]} + R(a,b)w \chi_{(a,b]} + R(b,x_n)w \chi_{(b,x_n]} .
\]
Thus
\[
g_n \nearrow f^0 \chi_{(0,a]} + R(a,b)w \chi_{(a,b]} + R(b,\infty)w \chi_{(b,\infty)} = f^0 \ \  \text{a.e.} ,
\]
and we conclude as in a). The proof is completed.
\endproof

Now we state the main theorem of this section.

\begin{theorem}
\label{th: level represent} Let $\varphi $ be an $N$-function and $W(\infty )=\infty $. Then for any $%
f=f^*\in \mathcal{M}_{\varphi ,w}$  we have
\begin{equation*}
P_{\varphi ,w}\left( f\right) =\int_{I}\varphi \Big( \frac{f^{0}}{w}\Big)
w=\int_{I}\varphi \Big( \frac{f}{w^{f}}\Big) w^{f}.
\end{equation*}
\end{theorem}

\begin{proof}
Let $f=f^{\ast }\in \mathcal{M}_{\varphi ,w}$ be arbitrary with $P_{\varphi
,w}\left( f\right) <\infty $. In view of Proposition \ref{prop: level represent} we assume that $\supp f = (0,\infty)$. Applying Lemma \ref{lem:degenerate}, $f$ does not have any degenerate level interval. Thus it remains to consider the following two cases.

 First suppose there is a sequence $s_{n}\nearrow \infty $ such that each $s_{n}$
is on the boundary of some m.l.i. of $f$. Define $g_{n}=f\chi _{\left[
0,s_{n}\right] }$. Then $g_{n}\nearrow f$ a.e. and by  Lemma 4.6 in \cite{KR12},
\begin{equation*}
P_{\varphi ,w}\left( g_{n}\right) \rightarrow P_{\varphi ,w}( f) .
\end{equation*}%
Moreover, for such chosen $(s_{n})$ each m.l.i. of $g_{n}$ is also m.l.i. of
$f$ and therefore we see that $g_{n}^{0}=f^{0}\chi _{\left[ 0,s_{n}\right] }$%
. Then $g_n^0 \nearrow f^0$ and by Proposition \ref{prop: level represent} and the Lebesgue Convergence Theorem,
\begin{equation*}
P_{\varphi ,w}(g_{n})=\int_{I}\varphi ( {g_{n}^{0}}/{w})
w\rightarrow \int_{I}\varphi ( {f^{0}}/{w}) w ,
\end{equation*}%
which gives the claim.

 Now assume there is $s$ such that each l.i. of $f$ is in $[0,s]$. Take $%
(s_{n})$ satisfying $s<s_{n}\nearrow \infty $ and put $g_{n}=f\chi _{\left[
0,s_{n}\right] }$. Then once again $g_{n}^{0}=f^{0}\chi _{\left[ 0,s_{n}%
\right] }$, because there is no l.i. of $g_{n}$ in $(s,s_{n})$, and we
conclude this case as above. The proof is completed.
\end{proof}

Summarizing main
results of sections 2 and 4 (especially Theorems \ref{th:01} and \ref{th: level
represent}) we get the following theorem.

\begin{theorem}
\label{final theorem}Let $w$ be a decreasing weight and $\varphi $ be $N$%
-function. Then the Köthe dual spaces to Orlicz-Lorentz spaces $\Lambda
_{\varphi ,w}$ and $\Lambda _{\varphi ,w}^{0}$ are expressed as
\begin{equation*}
\left( \Lambda _{\varphi ,w}\right) ^{\prime }=\mathcal{M}_{\varphi _{\ast
},w}^{0}\ \ \ \text{and}\ \ \ \left( \Lambda _{\varphi ,w}^{0}\right)
^{\prime }=\mathcal{M}_{\varphi _{\ast },w} ,
\end{equation*}%
with%
\begin{equation*}
\Vert f\Vert _{(\Lambda _{\varphi ,w})^{\prime }}= \|f\|^0_{\mathcal{M}%
_{\varphi _{\ast },w}}= \inf_{k>0}\Big\{ \frac{1}{k}\left(P_{\varphi_* ,w}\left(
kf\right) +1\right)\Big\} ,
\end{equation*}%
\begin{equation*}
\Vert f\Vert _{(\Lambda _{\varphi ,w}^{0})^{\prime }}= \|f\|_{\mathcal{M}%
_{\varphi _{\ast },w}}= \inf \{ \lambda >0:P_{\varphi_* ,w}\left( {%
f}/{\lambda }\right) \leq 1\} ,
\end{equation*}%
where
\begin{equation*}
P_{\varphi_* ,w}\left( f\right) = \inf \Big\{\int_I\varphi_*\left({f^*%
}/{|g|}\right) |g|: g\prec w\Big\}.
\end{equation*}
If in addition we assume that $W(\infty)=\infty$ for $I=[0,\infty)$ then we also have that
\begin{equation*}
P_{\varphi_* ,w}\left( f\right) = \int_{I}\varphi_* ( {%
(f^{\ast })^{0}}/{w}) w=\int_{I}\varphi_* ( {f^{\ast }}/{%
w^{f^{\ast }}}) w^{f^{\ast }},
\end{equation*}
where $(f^{\ast })^{0}$ is a level functions of $f^*$
with respect to $w$ and $w^{f^*}$ is an inverse level function of $w$ with
respect to $f^*$.
\end{theorem}

  For $\varphi (u)=\frac{1}{p}u^{p}$, $1<p<\infty$, we denote the space $\Lambda_{\varphi ,w}$ by $\Lambda_{p,w}$.
The next corollary provides an isometric description of $(\Lambda_{p,w})'$.
  The second formula recovers Halperin's Theorem 6.1 and Corollary on page 288 in \cite{Ha53}.

\begin{corollary}
\label{cor: Halperin}Let $1<p<\infty $ and $1/p+1/q=1$. Then for any $f\in
(\Lambda _{p,w})^{\prime }$ we have
\begin{equation*}
\Vert f\Vert _{(\Lambda _{p,w})^{\prime }}=\inf \Big\{ \Big( \int_{I}(
{f^{\ast }}/{|g|}) ^{q}|g|\Big) ^{{1}/{q}}:g\prec w\Big\} .
\end{equation*}
If in addition $W(\infty)= \infty$ in case of $I=[0,\infty)$, then
\begin{equation*}
\Vert f\Vert _{(\Lambda _{p,w})^{\prime }}=\Big( \int_{I}( {(f^*)^{0}
}/{w}) ^{q}w\Big) ^{{1}/{q}}.
\end{equation*}
\end{corollary}

\begin{proof}
The first equality follows from Theorem \ref{th:01}, while the second one from Theorem
\ref{th: level represent}.
\end{proof}

\begin{remark}
\rm{In Lorentz's paper \cite{Lor53} a theorem (Theorem 3.6.5) on duality of the space $\Lambda_{p,w}$ for $1<p<\infty$ was stated in terms of "level functions", however his definition of a level function is different from the one introduced earlier by Halperin. A similar notion of a level function has been later used by Sinnamon (see \cite{Si06}, Chapter 2.9). Both Lorentz's  and Halperin's representations suggest that ${f^0}/{w}=({f}/{w})^L$ for every non-negative and decreasing function $f$, where the right side means the level function of $f/w$ in the Lorentz sense.  It is straightforward to check this equality for a  characteristic decreasing function.}
\end{remark}

\section{\protect\bigskip Sequence case.}

We complete the discussion on the dual of Orlicz-Lorentz spaces considering here the discrete case. All results above for function spaces are valid in the Orlicz-Lorentz sequence spaces as well. Recall that for a given sequence $x=(x_i) $, its decreasing rearrangement $x^* = (x^*_i)$ is defined as $x^{\ast }_i=\inf \left\{ \lambda >0:d_{x}(\lambda )< i\right\}$, $i\in \mathbb{N}$, where
$d_{x}(\lambda )=\mu \left\{ i\in \mathbb{N}:\vert x_i\vert >\lambda
\right\}$ for $\lambda >0$, and $\mu $ is a counting measure.
Then given an Orlicz function $\varphi$ and a decreasing positive weight sequence $w = (w_i)$, the Orlicz-Lorentz sequence space $\lambda _{\varphi ,w}$ is defined as
\begin{equation*}
\lambda _{\varphi ,w}=\Big\{ x = (x_i)\in l^{0}:\exists _{\delta >0}\ \sum_{i=1}^{\infty }\varphi (\delta x^{\ast }_i)w_{i}<\infty \Big\} ,
\end{equation*}%
where $l^0$ is the space of real valued sequences.
We consider the space $\lambda _{\varphi ,w}$
with the Luxemburg norm $\Vert \cdot \Vert _{\lambda _{\varphi ,w}}$  denoted further
 by $\lambda _{\varphi ,w}$, or with the Amemiya norm $\Vert \cdot \Vert
_{\lambda _{\varphi ,w}}^{0}$ denoted  by $\lambda _{\varphi ,w}^{0}$. Those norms are defined analogously as for function spaces. The Orlicz-Lorentz sequence spaces are K\"othe spaces as subspaces of $l^0$, and their K\"othe dual spaces are defined analogously as in function case.
For each $x\in $ $\lambda _{\varphi ,w}$ we assign an element $\bar{x}\in\Lambda _{\varphi ,\bar{w}}$ on $[0,\infty )$, where
\begin{equation*}
\bar{x}=\sum_{i=1}^{\infty }x_{i}\chi _{\lbrack i-1,i)}\ \ \ \text{and}\ \ \ \bar{w}=\sum_{i=1}^{\infty }w_{i}\chi _{\lbrack i-1,i)} .
\end{equation*}%
The above correspondence between $x$ and $\bar{x}$ is a linear isometry between $\lambda_{\varphi,w}$ and a closed subspace of $\Lambda_{\varphi, \bar{w}}$.
Evidently
\begin{equation*}
\Vert x\Vert _{\lambda _{\varphi ,w}}=\Vert \bar{x}\Vert _{\Lambda _{\varphi
,\bar{w}}}\ \ \ \text{and }\ \ \ \Vert x\Vert _{\lambda _{\varphi ,w}}^{0}=\Vert \bar{x}%
\Vert _{\Lambda _{\varphi ,\bar{w}}}^{0} .
\end{equation*}%
The lemma below ensures that the respective correspondence remains true in the
dual space.

\begin{lemma}
\label{lem:5.1} Let $y = (y_i)\in (\lambda _{\varphi ,w})'$. Then
\begin{equation*}
\left\Vert y\right\Vert _{\left( \lambda _{\varphi ,w }\right) ^{\prime
}}=\left\Vert \bar{y}\right\Vert _{\left( \Lambda _{\varphi ,\bar{w}%
}\right) ^{\prime }}\text{ and }\left\Vert y\right\Vert _{\left( \lambda
_{\varphi ,w }^{0}\right) ^{\prime }}=\left\Vert \bar{y}\right\Vert
_{\left( \Lambda _{\varphi ,\bar{w}}^{0}\right) ^{\prime }}  .
\end{equation*}%
\end{lemma}

\begin{proof}
Define an averaging operator $T$ on $\Lambda _{\varphi ,\bar{w}}$ by
\begin{equation*}
T:h\rightarrow \sum_{i=1}^{\infty }\Big( \int_{[i-1,i)}h\Big) \chi
_{\lbrack i-1,i)} .  \label{aver oper}
\end{equation*}%
Then by \cite[Theorem 4.8]{BS88}, $\left\Vert Th\right\Vert _{\Lambda
_{\varphi ,\bar{w}}}\leq \left\Vert h\right\Vert _{\Lambda _{\varphi ,\bar{w}%
}}$ for each $h\in \Lambda _{\varphi ,\bar{w}}$. Moreover, for any $y\in
\left( \lambda _{\varphi ,w}\right) ^{\prime }$,
\begin{equation*}
\int_{0}^{\infty }\bar{y}h=\int_{0}^{\infty }\bar{y}(Th) .
\end{equation*}%
Therefore
\begin{eqnarray*}
\left\Vert \bar{y}\right\Vert _{\left( \Lambda _{\varphi ,\bar{w}%
}\right) ^{\prime }} &=&\sup \Big\{ \int_{0}^{\infty }\bar{y}h:\left\Vert
h\right\Vert _{\Lambda _{\varphi ,\bar{w}}}\leq 1\Big\}
=\sup \Big\{ \int_{0}^{\infty }\bar{y}(Th):\left\Vert h\right\Vert
_{\Lambda _{\varphi ,\bar{w}}}\leq 1\Big\} \\
&=&\sup \Big\{ \int_{0}^{\infty }\bar{y}(Th):\left\Vert Th\right\Vert
_{\Lambda _{\varphi ,\bar{w}}}\leq 1\Big\}
=\sup \Big\{ \int_{0}^{\infty }\bar{y}\bar{z}:\left\Vert z\right\Vert
_{\lambda _{\varphi ,w}}\leq 1\Big\} \\
&=&\sup \Big\{ \sum_{i=1}^{\infty }y_{i}z_{i}:\left\Vert z\right\Vert
_{\lambda _{\varphi ,w}}\leq 1\Big\} =\left\Vert y\right\Vert _{\left(
\lambda _{\varphi ,w }\right) ^{\prime }} .
\end{eqnarray*}
Similarly we prove the second equality.
\end{proof}
By analogy to function case the following space has been defined in \cite{KR12},
\begin{equation*}
\mathfrak{m}_{\varphi ,w}=\{ x\in l^{0}:\exists _{\lambda >0}\ p_{\varphi
,w}\left( x/\lambda \right) <\infty \} ,
\end{equation*}%
with the modular
\begin{equation*}
p_{\varphi ,w}\left( x\right) =\inf \Big\{ \sum_{i=1}^{\infty }\varphi
({x_{i}^{\ast }}/{|y_{i}|}) |y_{i}|:y\prec w\Big\}  ,
\end{equation*}%
where the submajorization of sequences $y\prec w$ means that
$\sum_{i=1}^{n}y_{i}^{\ast }\leq \sum_{i=1}^{n}w_{i}$ for all $n\in\mathbb{N}$. By $\mathfrak{m}_{\varphi
,w}$ denote the space equipped with the Luxemburg norm $\Vert \cdot \Vert _{%
\mathfrak{m}}$ and by $\mathfrak{m}_{\varphi ,w}^{0}$ the space endowed with the
Amemiya norm $\Vert \cdot \Vert _{\mathfrak{m}}^{0}$. Moreover, we adopt
definitions of the previous chapter to the sequence case setting for decreasing sequence $x = (x_i)$ and  $a,b\in\mathbb{N}\cup \{0\}$,
$a<b$,
\begin{equation*}
w\left( a,b\right) =\sum_{i=a+1}^{b}w_{i} ,\ \ \ x\left( a,b \right)
=\sum_{i=a+1}^{b}x_{i} ,\ \ \ r\left( a,b\right) =\frac{x\left( a,b\right) }{%
w\left( a,b\right) } .
\end{equation*}%
Then $(a,b]= \{a+1,\dots,b\}\subset\mathbb{N}$ is called a\textit{\ level interval of $x$ with respect to $w$} if for
each $j = a+1, \dots,b$,
\begin{equation*}
r\left( a,j\right) \leq r\left( a,b\right) \text{ and }0<r\left( a,b\right)  ,
\end{equation*}%
and the \emph{level sequence $x^{0}$ of $x$ with respect to $w$} is defined as
\begin{equation*}
x_{i}^{0}=\left\{
\begin{array}{cc}
r\left( a_{n},b_{n}\right) w_{i} & \text{for }i\in (a_{n},b_{n}] , \\
x_{i} & \text{otherwise} ,%
\end{array}%
\right.
\end{equation*}%
where $(a_{n},b_{n}]$ is a sequence of all maximal level intervals of $x$.
Notice that the results of the previous section ensure that the correspondence between $x$ and $\bar{x}$ preserves such defined level intervals. In fact (see the proofs of Lemmas \ref{lem:3}, \ref{Lemat ciag prostych}) we have for any $a\in\mathbb{N}\cup\{0\}$,
$r(a,j) \le r(a,b)$ for all $j= a+1,\dots,b$, if and only if $\bar{x}(a,t)/\bar{w}(a,t) \le \bar{x}(a,b)/\bar{w}(a,b)$ for all $t\in (a,b)$. Hence $(a,b]\subset \mathbb{N}$  is an m.l.i. of $x$ with respect to $w$ if and only if $(a,b)$ is an m.l.i. of $\bar{x}$ with respect to $\bar{w}$. Therefore
\begin{equation}\label{eq:5.1}
\int_0^\infty \varphi({(\bar{x}^*)^0}/{\bar{w}})\bar{w} = \sum_{i=1}^\infty \varphi({(x_i^*)^0}/{w_i}) w_i .
\end{equation}
Moreover, in view of Lemma \ref{lem:3}, $y\prec w$ if and only if $\bar{y} \prec \bar{w}$ and thus
\[
p_{\varphi,w}(x) = \inf\Big\{\int_0^\infty \varphi({\bar{x}^*}/{|\bar{y}|}) |\bar{y}|: \bar{y} \prec \bar{w}\Big\}.
\]
Hence by Lemma \ref{lem:2} applied to the step function $\bar{x}^*$ we obtain that
\begin{equation}\label{eq:5.2}
P_{\varphi, \bar{w}} (\bar{x}) = p_{\varphi,w}(x) .
\end{equation}
Finally, employing equalities (\ref{eq:5.1}), (\ref{eq:5.2}), Lemma \ref{lem:5.1} and Theorem \ref{final theorem}, we can state duality result for Orlicz-Lorentz sequence space $\lambda_{\varphi,w}$.

\begin{theorem}
\bigskip Let $w$ be a decreasing weight sequence and $\varphi $ be an
$N$-function. Then
\begin{equation*}
\left( \lambda _{\varphi ,w}\right) ^{\prime }=\mathfrak{m}_{\varphi_* ,w}^{0}%
\text{ and }\left( \lambda _{\varphi ,w}^{0}\right) ^{\prime }=\mathfrak{m}_{\varphi_* ,w} .
\end{equation*}%
If in addition $\sum_{i=1}^\infty w_i = \infty$, then
\begin{equation*}
p_{\varphi _{\ast },w}\left( x\right) =\sum_{i=1}^\infty\varphi
_{\ast }\Big( \frac{(x_i^{\ast })^{0}}{w_i}\Big) w_i .
\end{equation*}
\end{theorem}

\bigskip

\end{document}